\documentclass[twocolumn]{autart}    
\usepackage{picins}
\usepackage{amsmath,amssymb,amsfonts}
\usepackage{algorithmic}
\usepackage{graphicx}
\usepackage{textcomp}
\usepackage{stmaryrd}
\usepackage{graphics}
\usepackage{algorithm}
\usepackage{mathrsfs} 
\usepackage{epsfig}
\usepackage{indentfirst}
\usepackage{color}
\usepackage{epstopdf}
\usepackage{float}
\usepackage{cite}
\usepackage{amstext}
\usepackage{amsfonts}
\usepackage{booktabs}
\usepackage{flushend}
\usepackage{enumerate}  
\usepackage{multirow}
\usepackage{makecell}
\usepackage{pifont}
\usepackage{diagbox}         

\newtheorem{theorem}{Theorem}
\newtheorem{lemma}{Lemma}
\newtheorem{proposition}{Proposition}

\newtheorem{remark}{Remark}
\newtheorem{assumption}{Assumption}
\newenvironment{proof}{{\noindent \bf Proof.}\quad}{\hfill $\square$}
\usepackage{enumerate}
\usepackage[colorlinks=true,urlcolor=black]{hyperref}
\begin{document}
\begin{frontmatter}

\title{Distributed robust optimization for multi-agent systems with guaranteed finite-time convergence \thanksref{footnoteinfo}}
\thanks[footnoteinfo]{This paper was supported in part by the National Nature Science Foundation of China under Grant 61825301. Corresponding author Jun Fu.}

\author[China]{Xunhao Wu}\ead{neuwxh2102043@163.com},    
\author[China]{Jun Fu}\ead{junfu@mail.neu.edu.cn}              
\address[China]{State Key Laboratory of Synthetical Automation for Process Industries, Northeastern University, Shenyang 110819, China}                                                  

\begin{abstract}                          
  A novel distributed algorithm is proposed for finite-time converging to a feasible consensus solution satisfying global optimality to a certain accuracy of the distributed robust convex optimization problem (DRCO) subject to bounded uncertainty under a uniformly strongly connected network. Firstly, a distributed lower bounding procedure is developed, which is based on an outer iterative approximation of the DRCO through the discretization of the compact uncertainty set into a finite number of points. Secondly, a distributed upper bounding procedure is proposed, which is based on iteratively approximating the DRCO by restricting the constraints right-hand side with a proper positive parameter and enforcing the compact uncertainty set at finitely many points. The lower and upper bounds of the global optimal objective for the DRCO are obtained from these two procedures. Thirdly, two distributed termination methods are proposed to make all agents stop updating simultaneously by exploring whether the gap between the upper and the lower bounds reaches the certain accuracy. Fourthly, it is proved that all the agents finite-time converges to a feasible consensus solution that satisfies global optimality within a certain accuracy. Finally, a numerical case study is included to illustrate the effectiveness of the distributed algorithm.
\end{abstract}

\end{frontmatter}

\section{Introduction}
Multi-agent systems are network systems consisting of multiple decision-making agents, each possessing computational, communicative, learning, perceptual, and executive capabilities \cite{ferber1999multi}. Such systems have been used in a wide variety of fields, such as wireless networks \cite{kotary2020distributed,ghosal2020distributed}, power systems \cite{yin2022distributed,ufa2022review}, and robotics \cite{zhou2023racer,9851519}. To minimize the global cost by designing some suitable distributed controllers for the agents, distributed optimization for multi-agent systems has been extensively studied, leading to significant advancements in both theoretical and computational aspects \cite{yang2019survey,nedic2018distributed,zheng2022review}.
Constrained distributed optimization is one of the important categories since there may be various constraints, such as local constraints, global inequality and equality constraints in practical applications \cite{yang2019survey}.
\par
On the constrained distributed optimization, there are a considerable number of algorithms have been proposed, e.g., \cite{5404774,9335004,chen2020distributed,chen2021fixed,8262781,bastianello2022novel,xie2017distributed,xie2018distributed}. However, to the best of our knowledge, most of the existing constrained distributed optimization algorithms can only be applied to bi-directional (or undirected) and weight-balanced communication networks, except for literature \cite{xie2017distributed,xie2018distributed} which can be used to time-varying unbalanced directed graphs under the assumption of uniformly strong connectivity. Furthermore, these algorithms were designed for the case where the local data of all agents are completely accurate. However, these data of real-world optimization problems tend to be uncertain as a result of measurement/estimation errors and implementation errors \cite{ben2009robust}. Hence, the main focus of this article is to solve a distributed robust convex optimization problem (DRCO) with bounded uncertainty under the weakest assumption of network communication: uniformly strong connectivity \cite{burger2013polyhedral}.
\par
Recently, some distributed algorithms for dealing with the DRCO were developed in \cite{yang2008distributed,wang2016distributed,6461383,lee2015asynchronous,8022966,falsone2020scenario,you2018distributed,carlone2014distributed,chamanbaz2017randomized,chamanbaz2017,burger2012distributed,burger2013polyhedral,yang2014distributed}, which can be categorized into four groups according to the treatment of uncertainty. Firstly, inspired by the robust counterpart approach in \cite{ben2009robust}, some algorithms were proposed in \cite{yang2008distributed,wang2016distributed}, which make all the agents asymptotically converge to a feasible optimal solution of the DRCO by transforming the DRCO into a robust counterpart problem and then doing parallel computation with a constrained distributed optimization algorithm. However, these algorithms are confined to special constraint structures. Secondly, in \cite{6461383,lee2015asynchronous}, some random projection algorithms were designed that almost surely converge to a feasible optimal solution of the DRCO, yet the local feasibility of the solutions of all the agents cannot be guaranteed. Thirdly, some scenario-based algorithms were developed in \cite{8022966,falsone2020scenario,you2018distributed,carlone2014distributed,chamanbaz2017randomized,chamanbaz2017} by sampling a large number of scenarios from the uncertainty set to approximate the DRCO, which asymptotically converge to a probabilistically feasible approximate optimal solution. However, these algorithms can not converge to the feasible optimal solution of the DRCO. Fourthly, some most relevant algorithms to our article were presented in \cite{burger2013polyhedral,burger2012distributed,yang2014distributed,xunhao}. These algorithms are based on iteratively approximating the DRCO by populating the cutting-planes/cutting-surfaces into the existing finite sets of constraints. The algorithms given in \cite{burger2013polyhedral,burger2012distributed,yang2014distributed} asymptotically converge to a feasible optimal solution, while the algorithm in \cite{xunhao} enables all agents to finite-time converge to feasible and approximately optimal solutions. With the exception of \cite{xunhao}, to our best knowledge, none of the existing algorithms can guarantee the finite-time convergence and local feasibility of the solutions for all agents. However, the solutions in \cite{xunhao} only satisfy the zero-order optimality conditions and cannot provide specific accuracy assurance of global optimality. Therefore, the motivation of this article is to {\it propose a novel distributed algorithm for locating a feasible consensus solution satisfying global optimality to a certain accuracy of the DRCO under a uniformly strongly connected network within a finite number of iterations}.
\par
In this paper, the DRCO is studied for uniformly strongly connected multi-agent systems with the strictly convex global objective function. To solve this problem, based on the right-hand restriction approach \cite{Mitsos}, a distributed robust convex optimization algorithm is proposed, which has three parts. The first part is the distributed lower bounding procedure, which is based on iteratively approximating the DRCO by enforcing the compact uncertain sets at finitely many points. The second part is the distributed upper bounding procedure, which is developed by successively reducing the restriction parameters of the right-hand constraints and tightening the discretization of the compact uncertain sets. Both procedures above guarantee that each agent converges to the optimal solution of the DRCO. Moreover, at each iteration, the sum of the local objectives of all agents in the two procedures constitutes the lower and upper bounds of the global optimal value of the DRCO, respectively. The third part is an adaptation of the finite-time consensus algorithm proposed in \cite{xie2017stop}, which make all agents terminate simultaneously when the gap between the lower and the upper bounds reaches the certain accuracy. The main contribution of this paper is threefold:
\begin{enumerate}
	\item A distributed robust convex optimization algorithm is proposed to locate a feasible consensus solution of the DRCO satisfying global optimality to a certain accuracy under the assumption of a uniformly strongly connected network. 
	\item Two distributed termination methods are proposed to ensure finite-time convergence of the distributed robust convex optimization algorithm, and the performance of these two methods is compared.
	\item It is mathematically proven that the distributed robust convex optimization algorithm terminates within a finite number of iterations.
\end{enumerate} 
\par
The remaining sections of the paper are organized as follows. In Section 2, the problem formulation and some fundamental assumptions are given. In Section 3, two approximate problems of the DRCO are presented, and the distributed lower bounding procedure and distributed upper bounding procedure are designed. In Section 4, the distributed robust convex optimization algorithm for solving the DRCO is described, and the proof of finite-time convergence of the algorithm is also presented. In Section 5, a numerical case study is conducted to validate the proposed algorithm, and comparisons are provided between the proposed algorithm and some related algorithms. Finally, conclusions and an outlook on future work are drawn in Section 6.

\section{Problem Statement}
We consider a multi-agent system consisting of a set of agents $\mathcal{V}=\left\{1,... ,m\right\}$, in which each agent stores its local constraint, local cost function, identifier, and other private information. In a distributed optimization task, all the agents locate a feasible consensus solution for minimizing the global objective function based on agent communication and local computation. However, in real-world multi-agent systems, there exist perturbation errors in the parameters of each agent due to erroneous inputs, such as in the estimation and implementation. In order to guarantee the safety and stability of the system, we consider a distributed robust convex optimization problem with bounded uncertainty in the local objective functions and constraints of the following form to find a robust optimal consensus solution for the multi-agent system.
\begin{flalign} \label{DRO}
	\hspace{-11mm}
	\begin{aligned}
		&\mathop{\min}\limits_{x} &     & F(x)=\sum^m_{i=1} f_i(x,\delta_i)\\
		&\ \rm{s.t.}&     & x\in X= \mathop{\bigcap}\limits_{i=1}^m \mathop{\bigcap}\limits_{\gamma_i\in \Gamma_i} X_i(\gamma_i),
	\end{aligned}
\end{flalign}
where $x\in\mathbb{R}^n$ is a common decision vector of agents. $\delta_i\in \Delta_i$ and $\gamma_i\in \Gamma_i$ represent the uncertain parameters of local objective function and constraint set for agent $i$, respectively. $\Delta_i$ and $\Gamma_i$ are non-empty and compact sets. For each $i=1,...,m$, $f_i(\cdot):\mathbb{R}^n\times\Delta_i\rightarrow\mathbb{R}$ is the local objective function of agent $i$, and $X_i={\bigcap}_{\gamma_i\in \Gamma_i} X_i(\gamma_i)\subseteq\mathbb{R}^n$ is its constraint set. Suppose that $f_i$ is a convex function on $x$, and the set $X_i$ is convex and compact for all $i\in\mathcal{V}$. 
\par
Problem (\ref{DRO}) is the general form of a distributed robust convex optimization problem, where the robust optimal point $x^*\in X$ minimizes the global objective function $F$ considering the worst-case uncertainty $\overline\delta_i=\arg{\max}_{\delta_i\in\Delta_i}f_i$ for all $i\in\mathcal{V}$. 
To simplify the problem (\ref{DRO}), we adopt the epigraphic reformulation technique to transform (\ref{DRO}) into a standard distributed robust convex optimization problem. This technique ensures that the uncertainty is confined to the constraints, while the local objective functions remain unaffected \cite{ben2009robust}. The following is the standard form of a distributed robust convex optimization problem (\ref{DRCO}).
\begin{flalign} \label{DRCO}
	\tag{DRCO}
	\hspace{-1mm}
	\begin{aligned}
		&\mathop{\min}\limits_{x} &     & F(x)=\sum^m_{i=1}f_i(x)\\
		&\ \rm{s.t.}                            &     & g_i(x,y_i)\le0,\ \forall y_i\in Y_i,\quad i=1,2,...m,
	\end{aligned}
\end{flalign}
where $x\in\mathbb{R}^n$ is a common decision vector of agents, and $y_i \in Y_i$ is an uncertain vector of agent $i$. $Y_i\subseteq\mathbb{R}^{n_y}$ is a non-empty and compact set. For each $i=1,...,m$, $f_i:\mathbb{R}^n\rightarrow \mathbb{R}$ is the local objective function, and $g_i:\mathbb{R}^n\times Y_i\rightarrow \mathbb{R}$ is the local constraint function, where $g_i$ is a semi-infinite constraint consisting of finite-dimensional decision variables and an infinite number of inequality constraints. For all agents $i$, assume that $f_i$ and $g_i$ are convex functions concerning $x$ and that $g_i$ is continuous with respect to $y_i$. Let the feasible region of agent $i$ be $X_i=\left\{x|g_i(x,y_i)\le 0\ \forall y_i\in Y_i\right\}$.
\begin{assumption}[Solvability and Uniqueness]\label{Assumption 1}	
	The global feasible region of the (\ref{DRCO}) $X=\bigcap^m_{i=1} X_i$ is not empty, and
	the global objective function $F(x)$ is strictly convex on $x\in \mathbb{R}^n$, i.e., for any two point $u\neq v\in \mathbb{R}^n$, any $0<\theta<1$, there is
	\begin{flalign}
		F(\theta u+(1-\theta) v)<\theta F(u)+(1-\theta)F(v).
	\end{flalign} 
\end{assumption}
\par
Due to Assumption {\ref{Assumption 1}}, the (\ref{DRCO}) is solvable and has a unique optimal consensus solution $x^*$, i.e. for any $x\in X\verb|\| \left\{x^*\right\}$, it follows that $F^*=F(x^*)< F(x)$. Note that for the case where the global objective function is not strictly convex but convex, a tie-break rule can be used to ensure the uniqueness of the solution. For further interpretation of the tie-break rule, the reader is referred to \cite{burger2013polyhedral,Calafiore2006}.
\par
The communication among agents can be characterized in graph theory \cite{yang2019survey,west2001introduction}. A graph $\mathcal{G}=(\mathcal{V},\mathcal{E})$ can be used to represent the information sharing relationships among agents, where $\mathcal{V}=\left\{1,...,m\right\}$ denotes the vertex set, and $\mathcal{E}=\mathcal{V}\times\mathcal{V}$ is the edge set. A directed edge $(i,j)\in \mathcal{E}$ represents that agent $j$ can directly obtain information from agent $i$. $(i,i)$ indicates the self-loop of agent $i$. For agent $i$, the set of its in-neighbors is $N_i^{in}=\left\{j|(j,i)\in \mathcal{E}\right\}$, and the set of its out-neighbors is $N_i^{out}=\left\{j|(i,j)\in \mathcal{E}\right\}$. If for graph $\mathcal{G}$, $(j, i)\in\mathcal{E}$ if and only if $(i,j)\in\mathcal{E}$, then $\mathcal{G}$ is an undirected graph; otherwise, it is a directed graph. Define the weight matrix $A=\left\{a_{ij}\right\}\in \mathbb{R}^{m\times m}$, which satisfies that $a_{ij}>0$ if $j\in N_i^{in}$ and $a_{ij}=0$, otherwise. If $\sum^m_{i=1} a_{ij} =\sum^m_{i=1} a_{ji}$ for any $j\in \mathcal{V}$, then the graph $\mathcal{G}$ is weight-balanced; otherwise, $\mathcal{G}$ is weight-imbalanced. A time-invariant network is one in which the edge set remains unchanged over time slots, while a time-varying network experiences changes in the edge set due to unexpected loss of communication links. A time-invariant communication graph $\mathcal{G}$ is said to be strongly connected if and only if every agent in the graph is reachable from all other agents. For the time-variant network, the following assumption is commonly made \cite{yang2019survey,6930814}.
\begin{assumption}[Uniformly Strong Connectivity]\label{Assumption 2}
	The graph sequence $\left\{\mathcal G(t)\right\}$ is uniformly strongly connected, i.e. for all $t\ge 0$, there exists an interger $T>0$ that makes $G(t:t+T)$ strongly connected. 
	$$ \mathcal G(t:t+T)=([m],\mathcal G(t)\cup...\cup \mathcal G(t+T-1)).$$
\end{assumption}
\par
The assumption of uniformly strong connectivity is seen as the weakest assumption in network communication \cite{burger2013polyhedral}. The main focus of this article is to propose a distributed optimization algorithm that converges to an optimal consensus solution of the (\ref{DRCO}) in a finite number of iterations under the uniformly strongly connected assumption in the communication network while ensuring the feasibility of the solution at each agent for local constraints.
\section{Approximation Problems}
Since the constraints of all agents are semi-infinite constraints with finite-dimensional decision variables and an infinite number of inequality constraints, solving the (\ref{DRCO}) is NP-hard \cite{ben2009robust}. This section introduces the distributed lower bounding problem and designs a distributed lower bounding procedure based on the approach of successively tighter discretization of the compact sets $Y_i$. Then, this section presents the distributed upper bounding problem and illustrates that this problem is neither a relaxation nor a restriction of the (\ref{DRCO}). Moreover, a distributed upper bounding procedure is developed by successively reducing the restriction parameters of the right-hand constraints and tightening the discretization of the compact sets $Y_i$. 
	\subsection{Distributed Lower Bounding Procedure}
		\label{sec:Distributed Lower Bounding Procedure}
	A distributed lower bounding problem is introduced by discretizing the compact sets $Y_i$ into finite sets $\widetilde Y^k_i$.
	\begin{flalign} 
		\tag{$\mathrm{DLBD^k}$}\label{dlbd}
		\begin{aligned}
			&\mathop{\min}\limits_{x} &  & F(x)=\sum^m_{i=1}f_i(x)\\
			&\ \rm{s.t.}                             &  & g_i(x,y_i)\le0,\ \forall y_i\in \widetilde Y^k_i,\quad i=1,2,...m.
		\end{aligned}
	\end{flalign}
	Let $\widetilde X_i^k=\left\{x\in\mathbb{R}^n|g_i(x,y_i)\le0,\ \forall y_i\in \widetilde Y_i^k\right\}$ be the feasible set of the (\ref{dlbd}) for agent $i\in\mathcal{V}$. The global feasible set is $\widetilde X^k={\bigcap}_{i=1}^m\  \widetilde X^k_i$. For any finite set $\widetilde Y^{k+1}_i$, there is $\widetilde X_i^k\supset X_i$. Therefore, the (\ref{dlbd}) is a relaxation of the (\ref{DRCO}) and $\widetilde X^k\supset X$. Under Assumption 1, for any finite sets $\widetilde Y^{k+1}_i$ of all agents, the (\ref{dlbd}) is a solvable constrained distributed convex optimization problem, where the global optimal solution satisfies uniqueness. 
	\par
	Motivated by the strategy proposed in \cite{1976Infinitely,Mitsos,burger2013polyhedral}, we develop the first algorithmic primitive by successively tightening the discretization of the compact sets $Y_i$ of all agents, \emph {distributed lower bounding procedure}. 
	\par
    For any $k\ge0$, during the $(k+1)$-th iteration, each agent initially acquires the optimal point $x_i^{k+1}$ for (\ref{dlbd}) within finite time slots by agent communication and local computation. 
    Each agent then assigns the value of $x^{k+1}_i$ to $\widetilde x^{k+1}_i$ for the stopping criterion detection in Section \ref{sec:Finite-time Convergence}. Subsequently, each agent updates the finite set $\widetilde Y_i^{k+1}$ through the distributed lower bounding (DLBD) Oracle.
	\par
	\begin{itemize}
		\item[]
		\hspace{1em}
		{\bf DLBD Oracle} $\widetilde Y_i^{k+1}=LORC({x}^{k+1}_i,Y_i)$: verify the feasibility of a given point ${x}^{k+1}_i$ for agent $i$ by solving a lower level problem (LLP) to global optimality. 
		\begin{equation}\label{LLP}
			\tag{LLP}
			g^{\max}_i({x}^{k+1}_i)=\mathop{\max}\limits_{y_i\in Y_i} g_i({x}^{k+1}_i,y_i)
		\end{equation}
		\item[]
		\hspace{1em}
		{\bf If} $(\romannumeral 1)$ $g^{\max}_i({x}^{k+1}_i)>0$, i.e. ${x}^{k+1}_i\notin X_i$ then it populates the point $\hat{y}_i=\arg \max_{y_i\in Y_i} g_i({x}^{k+1}_i,y_i)$ into the finite set: $\widetilde Y_i^{k+1}=\widetilde Y_i^{k}\cup\left\{\hat{y}_i\right\}$, separating ${x}^{k+1}_i$ and $X_i$, {\bf otherwise} $(\romannumeral 2)$ it indicates that ${x}^{k+1}_i\in X_i$ and remains $\widetilde Y_i^{k+1}=\widetilde Y_i^{k}$. {\hfill $\square$}
	\end{itemize}
	\subsection{Distributed Upper Bounding Procedure}
	\label{sec:Distributed Upper Bounding Procedure}
	A distributed upper bounding problem is constructed by discretizing the compact sets $Y_i$ into finite sets $\overline Y_i^k$ and restricting the rights-hand constraints with proper positive parameters $\varepsilon^k_i$.
	\begin{flalign} \label{dubd}
		\tag{$\mathrm{DUBD^k}$}
		\begin{aligned}
			&\mathop{\min}\limits_{x} &  & F(x)=\sum^m_{i=1}f_i(x)\\
			&\ \rm{s.t.}                             &  & g_i(x,y_i)\le-\varepsilon_i^k,\ \forall y_i\in \overline Y_i^k,\quad  i=1,2,...m.
		\end{aligned}
	\end{flalign}
	Set $\overline X_i^k=\left\{x\in\mathbb{R}^n|g_i(x,y_i)\le-\varepsilon_i^k,\ \forall y_i\in \overline Y_i^k\right\}$ as the feasible set of the (\ref{dubd}) for agent $i\in\mathcal{V}$. The global feasible set is $\overline X^k={\bigcap}_{i=1}^m\  \overline X^k_i$.
	To elucidate the relations between the set $\overline X^k_i$ and the feasible set $X_i$ of the agent $i$ in the (\ref{DRCO}), we present the subsequent illustrations.
	\par
	\emph{Example 1}:
	Consider a distributed system where the constraint of the agent $i\in\mathcal{V}$ is a semi-infinite constraint as shown below: 
	\begin{equation}\label{example}
		\begin{aligned}
			& h_i(x,y_i)=[x(1)^2-2 x(1)]\times e^{-{x(1)}^2+{y_i}^2-2 x(1) y_i},\\
			&	g_i(x,y_i)=x(2)+{h_i(x,y_i)}\le 0, \ \ \forall y_i\in Y_i,
		\end{aligned}
	\end{equation}
	where $x=[x(1),x(2)]^\top\in \mathcal{F}_i=[0,2]\times[0,1]$, $Y_i=[0,2]$. Since $g_i(x,y_i)$ is a concave function on the uncertain parameter $y_i$, we can obtain that the semi-infinite constraint (\ref{example}) results in the feasible set of agents $i$: $X_i=\left\{x\in \mathcal{F}_i|x(2)+[x(1)^2-2\times x(1)]\times e^{-2{x(1)}^2}\le 0\right\}$.
	The corresponding inequality constraint constructed from the semi-infinite constraint is:
	\begin{equation}\label{inequality}
		g_i(x,y_i)=x(2)+{h_i(x,y_i)}\le -{\varepsilon^k_i}, \ \ \forall y_i\in \overline Y_i^k,
	\end{equation}
	where ${\varepsilon_i^k}\ge 0$, $\overline Y_i^k\subset Y_i$, and $\overline X_i^k$ is the feasible set of (\ref{inequality}).
	\par
	Figure \ref{neither} illustrates the relations between $X_i$ and $\overline X_i^k$ for different combinations of $\overline Y_i^k$ and $\varepsilon_i^k$.
	\begin{figure*}[t]
		\centering
		\begin{tabular}{cc}
			\begin{minipage}[t]{3in}
				\includegraphics[width=3in]{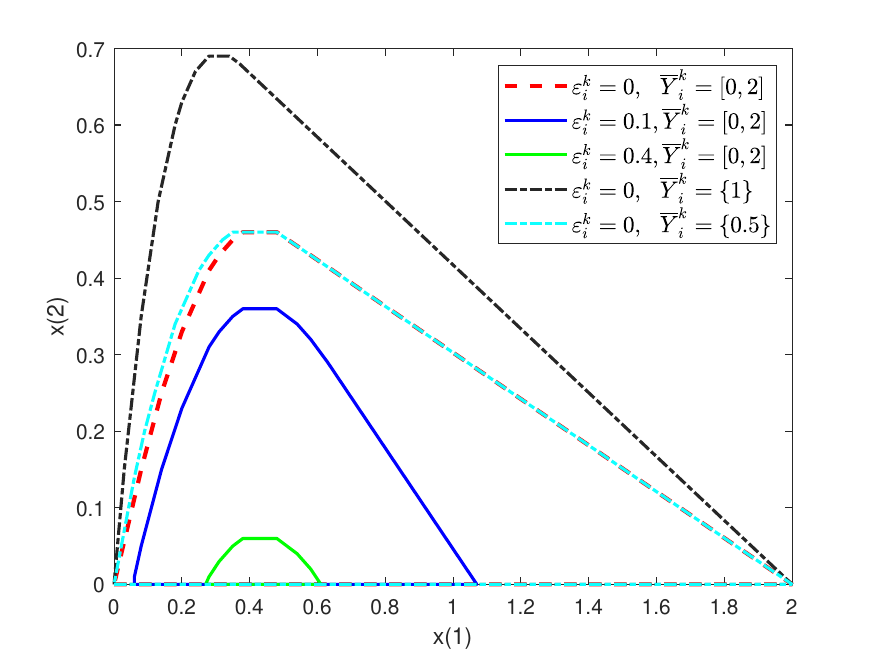}
			\end{minipage}
			\qquad
			\begin{minipage}[t]{3in}
				\includegraphics[width=3in]{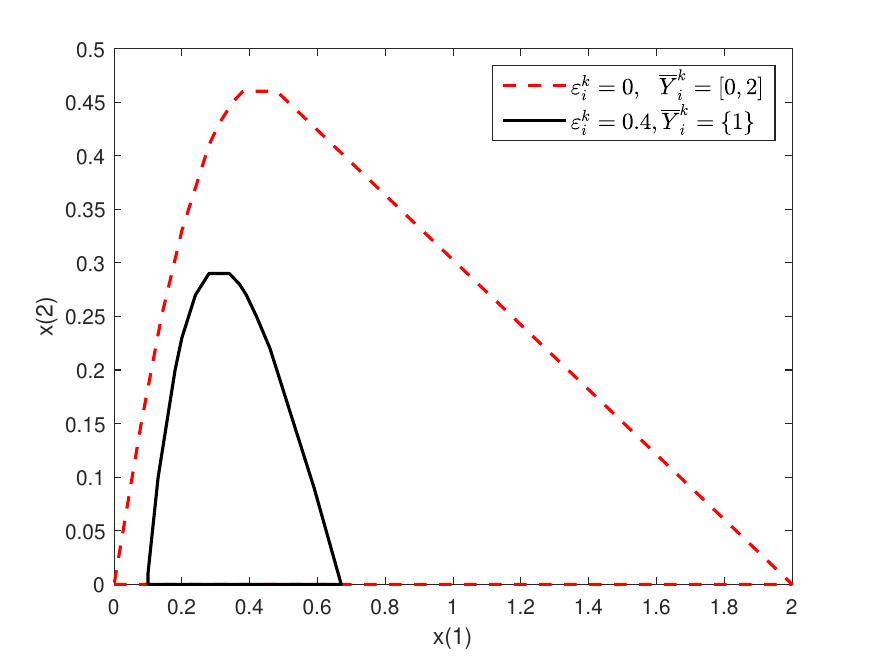}
			\end{minipage} 
			\\
			\begin{minipage}[t]{3in}
				\includegraphics[width=3in]{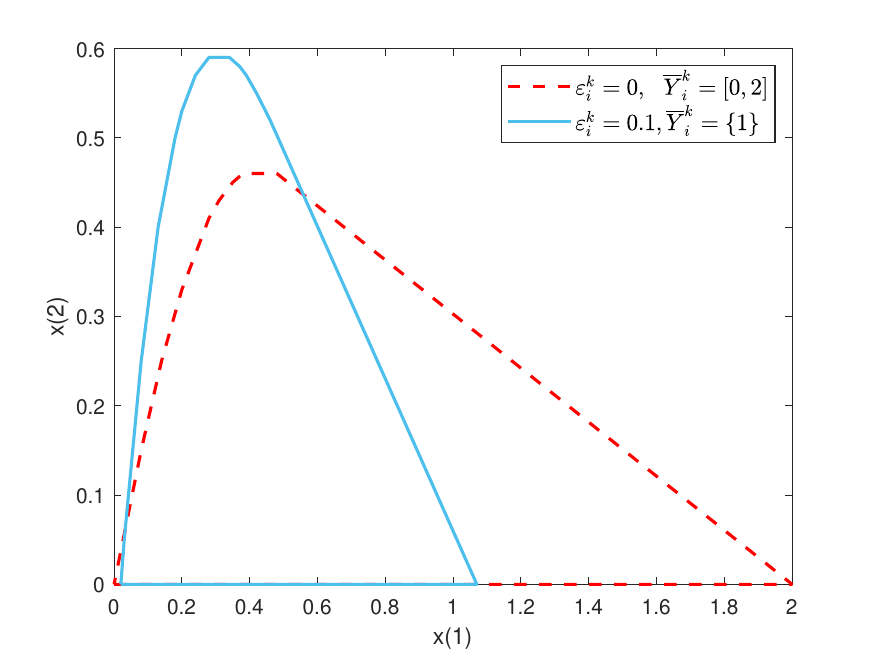}
			\end{minipage}
			\qquad
			\begin{minipage}[t]{3in}
				\includegraphics[width=3in]{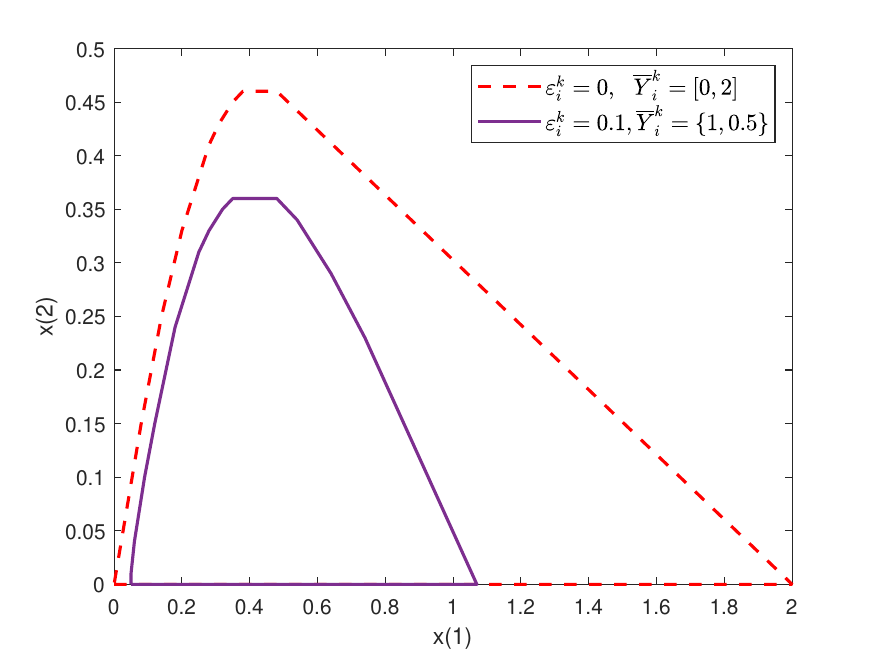}
			\end{minipage}
		\end{tabular}
		\caption{Graphic illustration of distributed upper bounding problem for $\mathcal{F}_i=[0,2]\times[0,1]$, $Y_i=[0,1]$ and $g(x,y_i)=x(2)+[x(1)^2-2\times x(1)]\times e^{-{x(1)}^2+{y_i}^2-2 x(1) y_i}$. The feasible region of the semi-infinite constraint is enclosed by the red dashed line. 
		The feasible regions of the inequality constraints, which are composed of different $\overline{Y}^k_i$ and $\varepsilon^k_i$, are enclosed by solid or dotted-dashed lines.}
		\label{neither}
	\end{figure*}
	\begin{enumerate}[a)] 
		\item For $\overline Y_i^k=Y_i$ and $\varepsilon_i^k>0$, the (\ref{dubd}) is a restriction of the (\ref{DRCO}): the feasible set of the agent $i$ satisfies $\overline X_i^k\subset X_i$; 
		\item For the finite set $\overline Y_i^k\subset Y_i^k$ and $\varepsilon_i^k=0$, the (\ref{dubd}) is a relaxation of the (\ref{DRCO}): the feasible sets of the agent $i$ satisfy $\overline X_i^k\supset X_i$;
		\item For $\varepsilon_i^k=0.4$ and $\overline Y_i^k=\left\{1\right\}$, the (\ref{dubd}) is a restriction of the (\ref{DRCO});
		\item For a smaller value $\varepsilon_i^k=0.1$ of the restriction parameter and the same finite set $\overline Y_i^k=\left\{1\right\}$, the (\ref{dubd}) is neither a relaxation nor a restriction of the (\ref{DRCO}): whether the optimal point of the (\ref{dubd}) is feasible for (\ref{DRCO}) relies on the global objective function $F$; 
		\item For $\varepsilon_i^k=0.1$ and $\overline Y_i^k=\left\{1,0.5\right\}$, the (\ref{dubd}) is again a restriction of the (\ref{DRCO}). This case is a tighter restriction compared to the case c).
	\end{enumerate}
	\par
	Based on the above example, it can be concluded that for any finite sets $\overline Y_i^k\subset Y_i^k$ and positive restriction parameters $\varepsilon_i^k$ of all agents, the (\ref{dubd}) cannot be regarded as a relaxation or a restriction of the (\ref{DRCO}). However, by gradually populating some points in the finite set $\overline Y_i^k$ and proportionally reducing the restriction parameter $\varepsilon_i^k$ for all agents, the feasible region of (\ref{dubd}) can gradually approach that of (\ref{DRCO}), which can make the optimal solution of (\ref{dubd}) converge to that of (\ref{DRCO}). Following this idea and the centralized right-hand restriction strategy \cite{Mitsos}, we propose the second algorithmic primitive: \emph {distributed upper bounding procedure}.
	\par
	For any $k\ge 0$, during the $(k+1)$-th iteration, each agent firstly obtains the optimal point $z_i^{k+1}$ of the (\ref{dubd}) in a finite number of time slots by exchanging local information with neighboring agents and performing local calculations. Then, each agent executes the distributed upper bouding (DUBD) Oracle via internal computation. Here, the vector $\overline x_i^{k+1}$ is chosen for the stopping criterion detection introduced in Section \ref{sec:Finite-time Convergence}.
	\begin{itemize}
		\item[]
		\hspace{1em}
		{\bf DUBD Oracle} $[\overline x_i^{k+1},\overline Y_i^{k+1},\varepsilon^{k+1}_i]=UORC$\\$(z_i^{k+1}, Y_i)$: queried at a given point $z_i^{k+1}$ for the compact set $Y_i$ by solving a lower level problem (LLP) to global optimality. 
		\begin{equation}
			\tag{LLP}
			g^{\max}_i(z_i^{k+1})=\mathop{\max}\limits_{y_i\in Y_i} g_i(z_i^{k+1},y_i)
		\end{equation}
		\item[]
		\hspace{1em}
		{\bf If} $(\romannumeral 1)$ $g^{\max}_i(z_i^{k+1})>0$, i.e. $z_i^{k+1}\notin X_i$ then it populates the point $\hat{y}_i=\arg \max_{y_i\in Y_i} g_i(z_i^{k+1},y_i)$ into the finite set: $\overline Y_i^{k+1}=\overline Y_i^k\cup\left\{\hat{y}_i\right\}$, separating $z_i^{k+1}$ and $X_i$. Assign a value to $\overline x^{k+1}_i$ that satisfies $f_i(\overline x^{k+1}_i)\rightarrow+\infty$, and let $\varepsilon_i^{k+1}=\varepsilon_i^k$, {\bf otherwise} $(\romannumeral 2)$ it reduces the restriction parameter proportionally, i.e. $\varepsilon_i^{k+1}\leftarrow\varepsilon_i^k/r$, where $r>1$ is a reduction parameter, and let $\overline Y_i^{k+1}=\overline Y_i^{k}$. Additionally, we assign the value of vector $z_i^{k+1}$ to $\overline x^{k+1}_i$. 
		{\hfill $\square$}
	\end{itemize}
	\par
	There may be an unsolvable case of the (\ref{dubd}) in the distributed upper bouding procedure, i.e. $\overline X^k=\emptyset$. We make the following assumption to exclude this case.
	\begin{assumption}[Interior Point]\label{Assumption 3}	
		The feasibile set of the (\ref{DRCO}) $X=\bigcap_{i=1}^m X_i$ contains at least one interior point $x_0$, i.e. 
		$$	g_i(x_0,y_i)<0,\ \forall y_i\in Y_i, \quad i=1,...,m.$$	
	\end{assumption}
	\begin{lemma}
		Under Assumption \ref{Assumption 3}, for any agent $i\in\mathcal{V}$, there exists an initial value $\varepsilon^{0}_i>0$ for the restriction parameter $\varepsilon_i^k$ such that the (\ref{dubd}) is solvable in any iteration $k+1>0$ of the distributed upper bounding procedure.
	\end{lemma}
	\begin{proof}
		Since the feasible set of the ({\ref{DRCO}}) exists at least an interior point $x_0$ that satisfies
		\begin{equation} \nonumber
			x_0\in \bigcap^m_{i=1} \left\{x\in \mathbb{R}^n|g_i(x,y_i)<0,\forall y_i\in Y_i\right\}.
		\end{equation}
		Hence, there are a set of $\varepsilon^{0}_i$ that satisfies
		\begin{equation}\nonumber
			g_i(x_0,y_i)+\varepsilon^{0}_i\le0,\ \forall y_i\in Y_i, \quad i=1,...,m.
		\end{equation}
		We set $\hat{X}=\left\{x\in\mathbb{R}^n|g_i({x},y_i)\le-\varepsilon^{0}_i,\ \forall y_i\in Y_i,\ \forall i\in\mathcal{V}\right\}$. The set $\overline{X}^k$ is the global feasible set of the (\ref{dubd}), i.e. $$\overline{X}^{k}=\left\{x\in\mathbb{R}^n|g_i({x},y_i)\le-\varepsilon^{k}_i,\ \forall y_i\in \overline Y_i^k,\ \forall i\in\mathcal{V}\right\}.$$
		As $\varepsilon^{0}_i\ge\varepsilon_i^k$ and $\overline Y_i^k\subset Y_i$ hold for any agent $i$, we have $\hat X \subseteq \overline{X}^k$ and $\overline{X}^k\neq\emptyset$. Therefore, the (\ref{dubd}) is solvable in any iteration of the distributed upper bounding procedure.
	\end{proof} 
    \par
    Therefore, under Assumptions \ref{Assumption 1} and \ref{Assumption 3}, take proper values of $\varepsilon_i^0$ for all agents, the (\ref{dubd}) is a solvable constrained distributed convex optimization problem in any iteration $k+1>0$ of the distributed upper bounding procedure, where the global optimal solution meets uniqueness.
	\par
	Overall, this section presents two approximation problems for the (\ref{DRCO}) and proposes a distributed lower bounding procedure and a distributed upper bounding procedure. The convergence of these procedures is given in Section \ref{sec:Algorithm Design and Convergence Analysis}.
	\section{Algorithm Design and Convergence Analysis}
	\label{sec:Algorithm Design and Convergence Analysis}
	In this section, based on the aforementioned procedures, a distributed robust convex optimization algorithm to locate a feasible consensus solution satisfying global optimality to a certain accuracy is described, and its finite-time convergence is established.
	\subsection{Distributed Robust Convex Optimization Algorithm}
	\label{{sec:Distributed Robust Convex Optimization Algorithm}}
	The algorithm for solving the (\ref{DRCO}) is as follows:
\begin{algorithm}[h]
	\caption{Distributed robust convex optimization algorithm}
	\label{alg:1}
	\vspace{1mm}\small
	{\textbf {Input:}}  for each agent $i\in\mathcal{V}$: initial restriction parameter $\varepsilon_i^{0}>0$; two finite or empty subsets of $Y_i$, namely $\widetilde Y_i^{0}$ and $\overline Y_i^{0}$; reduction parameter $r>1$; termination parameter $\epsilon^f$; iteration counter $k=0$.
	\vspace{.5em}
	\\
	{\textbf {Repeat:}}\vspace{-.5em}
	\begin{itemize}
		\item[$\bullet$] {\bf\textit {Distributed lower bounding procedure}}
		\begin{itemize}\vspace{-1.25em}
			\item[$\,$] 
			\item[$1:$] { Solve the (\ref{dlbd}) to optimality:} 
			each agent obtains the optimal solution $x_i^{k+1}$, and set $\widetilde x^{k+1}_i\leftarrow x^{k+1}_i$.
			\item[$2:$] { Call the DLBD Oracle for the compact set $Y_i$ at the query point $x^{k+1}_i$, i.e. $\widetilde Y_i^{k+1}=LORC(x^{k+1}_i,Y_i)$.}
		\end{itemize}
		\vspace{-.5em}
		\item[$\bullet$] {\bf\textit {Distributed upper bounding procedure}}
		\begin{itemize}\vspace{-1.25em}
			\item[$\,$] 
			\item[$3:$] { Solve the (\ref{dubd}) to optimality:} 
			each agent obtains the optimal solution $z_i^{k+1}$.
			\item[$4:$] {Call the DUBD Oracle for the compact set $Y_i$ at the query point $z_i^{k+1}$, i.e. $[\overline x_i^{k+1},\overline Y_i^{k+1},\varepsilon^{k+1}_i]=UORC(z_i^{k+1},Y_i)$.}		
	\end{itemize}
    \vspace{-.5em}
	\item[$\bullet$] {\bf\textit {Finite-time termination}}
	\begin{itemize}\vspace{-1.25em}
	\item[$\,$] 
	\item[$5:$] {Check whether the stopping criterion is satisfied (see: Section \ref{sec:Finite-time Convergence}).} 
	\item[$6:$] {\textbf{if}} {the stopping criterion is satisfied} {\textbf{then}}	
	\item[$7:$] 
	\begin{itemize}
		\item [$\ast$] Set the optimal solution of the (\ref{DRCO}) $x^{opt}_i\leftarrow \overline x^{k+1}_i$, {\bf{Terminate}}.
	\end{itemize}	 
	\item[$8:$] {\textbf{end if}}
	\item[$9:$] {\textbf{Set}} $k\leftarrow k+1$
	\end{itemize}
	\end{itemize}
\end{algorithm}
	\begin{remark}
	 In Steps 1 and 3 of the Algorithm \ref{alg:1}, we require to solve the (\ref{dlbd}) and (\ref{dubd}) to optimality within a finite number of time slots. In the literature, many approaches have been proposed to finite-time/fixed-time converges to a consensus solution in the discrete-time setting
	 \cite{rikos2022finite,chen2023privacy,jiang2022fast,yao2018distributed,prakash2019distributed,manitara2014distributed,mai2017local,carlone2014distributed,xie2017stop} or the continuous-time setting \cite{wang2020distributed,shi2020continuous,shi2022finite,garg2020fixed,song2021fixed,wu2021designing,shi2022finite,firouzbahrami2022cooperative}. Since the (\ref{dlbd}) and the (\ref{dubd}) are distributed convex optimization problems with local inequality constraints under uniformly strongly connected networks, we present two practical strategies for such problems to make all agents converge to their consensus optimal solutions within a finite number of time slots. The first strategy involves the D-RFP algorithm \cite{xie2018distributed} equipped with the finite-time consensus algorithm \cite{xie2017stop}. The second strategy entails transforming the (\ref{dlbd}) (or the (\ref{dubd})) into a distributed optimization problem with identical local objective functions via epigraphic reformulation (see: Example 2 in \cite{you2018distributed}), and then adopts the approach by exchanging the parameters of the local constraint function with neighboring agents to make all agents converge to an optimal consensus solution in a fixed number of time slots \cite{carlone2014distributed}.
    \end{remark}
	\begin{assumption}\label{Assumption 4}
		For distributed convex optimization problems with a finite number of inequality constraints under uniformly strongly connected networks, the decision variables of all agents $i\in\mathcal{V}$ can converge to an optimal consensus solution within finite time slots.
	\end{assumption}
\begin{remark}
	In Steps 2 and 4, it is necessary to globally solve the LLP problems. For the case that the constraint function $g_i$ is a differentiable and concave function with respect to $y_i$ for any agent $i\in\mathcal{V}$. According to the optimality condition of convex problems (see: Literature \cite{boyd2004convex} p267), we can obtain the optimal solution of the LLP problems by finding points that satisfy the Karush-Kuhn-Tucker (KKT) conditions. Literature \cite{ben2009robust,burger2012distributed} concludes some results of the solution to convex LLP problems under specific uncertain sets $Y_i$ and constraint functions $g_i$.
	For the case where the LLP is a nonconvex optimization problem, there is no direct method to find the global optimal solution. The two main indirect methods are the discretization method \cite{chen2005inequality} and the $\alpha BB$ method \cite{stein2012adaptive}, respectively. The former is based on iteratively approximating the LLP problems by successively discretizing the set $Y_i$. The latter focuses on adaptively constructing convex relaxations of the LLP problems.
\end{remark}
	\begin{assumption}\label{Assumption 5}
		For $i\in\mathcal{V}$, at any iteration $k+1>0$, the LLP is globally solved for the query point $\hat{x}_i$ either establishing $\max_{y_i\in Y_i}g_i(\hat{x}_i,y_i)\le 0$, or furnishing a point $\hat{y}_i$ such that $g_i(\hat{x}_i,\hat{y}_i)> 0$.
	\end{assumption}
\begin{lemma}\label{DLBD_lemma}
	For any $i\in\mathcal{V}$, take any $\widetilde Y^{0}_i\subset Y_i$. Under Assumptions \ref{Assumption 1}-\ref{Assumption 2} and \ref{Assumption 4}-\ref{Assumption 5}, suppose that ${x}^{k+1}$ is the optimal consensus solution of the (\ref{dlbd}) in the $(k+1)$-th iteration of the distributed lower bounding procedure, where ${x}_i^{k+1}=x^{k+1},\ \forall i\in\mathcal{V}$. Let $F^*$ be the optimal objective of the (\ref{DRCO}). Then,
	\begin{itemize}
		\item [\romannumeral 1)] $F(x_i^{k+1})\le F^*$ for all $i\in\mathcal{V}$ and all $k\ge 0$;
		\item [\romannumeral 2)] $F(x_i^{k+1})\le F(x_i^{k+2})$ for all $i\in\mathcal{V}$ and all $k\ge 0$;
		\item [\romannumeral 3)] There exists a point $p_i\in\mathbb{R}^n$ that the sequence $\left\{x_i^{k+1}\right\}_k$ converges to it, i.e.  $\mathop{\lim}\limits_{k\rightarrow +\infty} x_i^{k+1}=p_i $ for all $i\in\mathcal{V}$.
		\item [\romannumeral 4)] The limit point $p_i$ is feasible for the agent $i\in\mathcal{V}$ in the (\ref{DRCO}), i.e. $p_i \in X_i$.
	\end{itemize}
\end{lemma}
\begin{proof}
	{\it Proof of \romannumeral 1):} The feasible region $X$ of the (\ref{DRCO}) satisfies $$X=\mathop{\bigcap}\limits_{i=1}^m X_i\subset X_i.$$ For any agent $i\in\mathcal{V}$ and $k\ge 0$, it follows that $$\widetilde X_i^k=\left\{x|g_i(x,y_i)\le 0, \ \forall y_i\in \widetilde Y_i^k\right\}\supset X_i.$$ Since $\widetilde X^k=\bigcap_{i=1}^m \widetilde X_i^k\supset X$ for all $k\ge 0$, it can be concluded that $F(x^{k+1})\le F^*$. Therefore, it is satisfied that $F(x_i^{k+1})\le F^*$ for any $i\in\mathcal{V}$ and $k\ge 0$.
	\par
	{\it Proof of \romannumeral 2):} For any $k\ge 0$ and any $i\in\mathcal{V}$, after the step of the DLBD Oracle, there is
	$$\widetilde Y_i^k\subset \widetilde Y_i^{k+1}.$$
	Therefore, the feasible domain of (\ref{dlbd}) in two iterations satisfies the following relation:
	$$\widetilde X_i^k\supset \widetilde X_i^{k+1},\ \ \forall i\in\mathcal{V}.$$
	$x_i^{k+1}$ and $x_i^{k+2}$ are the solutions obtained from these two iterations of the distributed lower bouding procedure respectively. Therefore, we can conclude that $F(x_i^{k+1})\le F(x_i^{k+2})$ for any $i\in\mathcal{V}$ and $k\ge 0$.
	\par
	{\it Proof of \romannumeral 3):} 
	According to [Lemma \ref{DLBD_lemma}, \romannumeral 1)] and [Lemma \ref{DLBD_lemma}, \romannumeral 2)], we can conclude that the sequence $\left\{F(x_i^{k+1})\right\}_k$ is bounded and non-decreasing. It can be followed that the sequence $\left\{F(x_i^{k+1})\right\}_k$ is convergent, i.e.
	$$\mathop{\lim}\limits_{k\rightarrow +\infty} \vert F(x_i^k)-F(p_i)\vert= 0.$$
	Due to strict convexity of global objective function $F$, the sequence $\left\{x_i^{k+1}\right\}_k$ converges to a point $p_i$.
	\par
	{\it Proof of \romannumeral 4):} Refering to the proof idea in \cite{Mitsos}. we prove the feasibility of the limit point $p_i$ for the agent $i$. For $i\in\mathcal{V}$ and $k\ge 0$, consider the corresponding solution of the (\ref{LLP}) $\hat y_i$ in the DLBD Oracle for which $g_i(x_i^{k+1},\hat y_i)>0$. By reconstruction of the (\ref{dlbd}) we have
	\begin{align}\nonumber
		g_i(x_i^l,\hat y_i)\le 0,\quad \forall l>k+1> 0.
	\end{align}
	Since $g_i$ is a continuous function on $x$, for any $\epsilon>0$, there is a positive parameter $\delta$ satisfying: 
	\begin{align}\label{lemma1_3_1}
		g_i(x,\hat y_i)<\epsilon,\quad \Vert x-x_i^l\Vert<\delta,\ \forall l>k+1> 0.
	\end{align}
	Due to the convergence of $\left\{x_i^{k+1}\right\}_k$ [Lemma \ref{DLBD_lemma} \romannumeral 3)], we have for any $\delta>0$,
	\begin{equation}\label{lemma1_3_2}
		\exists K:\ \Vert x_i^l-x_i^{k+1}\Vert<\delta,\quad \forall l>k+1\ge K. 
	\end{equation}
	Combining the results of (\ref{lemma1_3_1}) and (\ref{lemma1_3_2}), we can obtain that for any $\epsilon>0$,
	\begin{align}\nonumber
		\exists K:\ g_i(x_i^{k+1},\hat y_i)<\epsilon,\quad \forall k+1\ge K. 
	\end{align}
	Since $g_i(x_i^{k+1},\hat y_i)>0$, we have $g_i(x_i^{k+1},\hat y_i)\rightarrow 0$. Therefore, the limit point $p_i$ is feasible:
	\begin{equation}\nonumber
		\mathop{\max}\limits_{y_i\in Y_i} g_i(p_i,y_i)=\mathop{\lim}\limits_{k\rightarrow \infty} g_i({x}_i^{k+1},\hat y_i)=0.
\vspace{-8mm}	
\end{equation}
\end{proof}
	\begin{proposition}
		For any $i\in\mathcal{V}$, take any $\widetilde Y^{0}_i\subset Y_i$. Suppose that Assumptions 1-2 and 4-5 hold. Then, the distributed lower bounding procedure is convergent, i.e. ${\sum}^m_{i=1}f_i({\widetilde x^{k+1}_i})\rightarrow F^*$. Moreover, for any $k\ge0$, it is satisfied that ${\sum}^m_{i=1}f_i({\widetilde x^{k+1}_i})\le F^*$.
	\end{proposition}
	\begin{proof}
		Under Assumption \ref{Assumption 4}, the following relation holds for any agent $i\in\mathcal{V}$: $\widetilde x^{k+1}_i=x^{k+1}_i=x^{k+1}$, owing to the consensus of the (\ref{dlbd}) solution. Combining the result of [Lemma \ref{DLBD_lemma} \romannumeral 1)], we have
		\begin{equation}
			\mathop{\sum}\limits^m_{i=1}f_i(\widetilde x_i^{k+1})=\mathop{\sum}\limits^m_{i=1}f_i(x_j^{k+1})=F(x_j^{k+1})\le F^{*},\ \forall k\ge0.
		\end{equation}
	    where $j$ is the identify of an arbitrary agent satisfying $j\in\mathcal{V}$.\\
		According to [Lemma \ref{DLBD_lemma}, \romannumeral 3)] and [Lemma \ref{DLBD_lemma}, \romannumeral 4)], we can conclude that the sequence $\left\{\widetilde x_i^{k+1}\right\}_k$ converges to a consensus point $p$ for any agent $i\in\mathcal{V}$, where $p$ is a feasible point for the (\ref{DRCO}), i.e. $p
		\in{\bigcap}_{i=1}^m X_i=X$.
		It follws that
		\begin{equation}\nonumber
			F(p)\ge F^{*}.
		\end{equation}
		Since $$F(p)=\lim_{k\rightarrow\infty}\mathop{\sum}\limits^m_{i=1}f_i(\widetilde x_i^{k+1})\le F^{*},$$ we can obtain $F(p)= F^{*}$. To sum up, the sequence $\left\{{\sum}^m_{i=1}f_i(\widetilde x_i^{k+1})\right\}_k$ converges to $F^*$.
	\end{proof}

\begin{lemma}\label{DUBD_lemma}
	For all agents $i\in\mathcal{V}$, take any $\overline Y^0_i\subset Y_i$ and proper restriction parameter $\varepsilon^0_i$. Suppose that Assumptions \ref{Assumption 1}-\ref{Assumption 5} hold. Let ${z}^{k+1}$ be the optimal consensus solution of the (\ref{dubd}) in the $(k+1)$-th iteration of the distributed upper bounding procedure, in which ${z}_i^{k+1}=z^{k+1},\ \forall i\in\mathcal{V}$. Then, 
	\begin{itemize}
		\item [\romannumeral 1)] there exists a point $q_i$ that the sequence $\left\{z_i^{k+1}\right\}_k$ converges to it, i.e. $\mathop{\lim}\limits_{k\rightarrow+\infty}z_i^{k+1}=q_i$ for all $i\in\mathcal{V}$.
		\item [\romannumeral 2)]
		any agent $i$ can obtain a locally feasible point $\hat{x}_i$ for the (\ref{DRCO}) (i.e. $\hat{x}_i\in X_i$) within a finite number of iterations through the distributed upper bounding procedure.
	\end{itemize}	
\end{lemma}
\begin{proof}
	{\em {proof of \romannumeral 1):}}
	Before proving this result, we make the following settings:
	\begin{flalign}\label{8}
		\begin{aligned}
			&\overline X^k=\mathop{\bigcap}\limits_{i=1}^m\left\{x\in \mathbb{R}^n|g_i(x,y_i)\le -\varepsilon^k_i, \ \forall y_i\in \overline Y_i^{k}\right\}, \\
			&	{A}^k=\mathop{\bigcap}\limits_{i=1}^m\left\{x\in \mathbb{R}^n|g_i(x,y_i)\le -\varepsilon_i^k, \ \forall y_i\in Y_i\right\},\\
			&	{B}^k=\mathop{\bigcap}\limits_{i=1}^m\left\{x\in \mathbb{R}^n|g_i(x,y_i)\le 0, \ \forall y_i\in \overline Y^k_i\right\}.
		\end{aligned}		
	\end{flalign}
	Based on the above settings, we can easily derive the following result: for any $k\ge 0$,
	\begin{flalign}
		{A}^k\subset \overline X^k\subset B^k.
	\end{flalign}
	As the number of iterations of the distributed upper bounding procedure increases, the value of $\varepsilon_i^k$ decreases proportionally, while the number of elements in the set $\overline Y^k_i$ continues to increase. It follows that
	\begin{flalign}\label{10}
		\begin{aligned}
			{A}^k\subset{A}^{k+1}\subset\cdots\subset{A}^{+\infty}=X,\\
			{B}^k\supset{B}^{k+1}\supset\cdots\supset{B}^{+\infty}=X.
		\end{aligned}
	\end{flalign}
	According to the squeeze theorem, it can be concluded that
	\begin{equation}
		\mathop{\lim}\limits_{k\rightarrow \infty}  \overline X^k\rightarrow X.
	\end{equation}
	Since the global objective function is strictly convex and the global feasible region in the (\ref{dubd}) converges to that of the (\ref{DRCO}), the optimal point sequence $\left\{z_i^{k+1}\right\}_k$ of the (\ref{dubd}) is convergent for any agent $i\in\mathcal{V}$.
	\par
	{\em {proof of \romannumeral 2):}}
	Since the (\ref{dubd}) is neither a restriction nor a relaxation of the (\ref{DRCO}), the solution of the (\ref{dubd}) may not lie in the feasible domain of the (\ref{DRCO}). Referring to the proof idea in \cite{Mitsos}, we prove that the solution of the (\ref{dubd}) satisfying the constraints in the (\ref{DRCO}) can be obtained within a finite number of iterations of the distributed upper bounding procedure. For any agent $i\in\mathcal{V}$, suppose that at the $(k+1)$-th iteration of the distributed upper bounding procedure, the optimal point of agent $i$ obtained by solving the (\ref{dubd}) does not satisfy its local semi-infinite constraint, i.e. $z_i^{k+1}\notin X_i$, and let $\hat{y}_i$ be the corresponding maximum constraint violation point obtained through the DUBD Oracle. It follows that
	$$g_i(z_i^{k+1},\hat{y}_i)>0.$$
	For any $l>k+1> 0$, there exists a positive parameter $\epsilon$ that satisfies
	$$g_i(z_i^l,\hat{y}_i)\le -\epsilon<0.$$
	Given that $X_i$ and $Y_i$ are both compact sets and the constraint function $g_i(x,y_i)$ is continuous, we can conclude that for any $l>k+1> 0$
	\begin{equation}
		\exists \delta>0,\ \  g_i(x,\hat{y}_i)\le -\epsilon/2<0,\ \ \Vert x-z_i^l\Vert<\delta. 
	\end{equation}
	Since the sequence $\left\{z_i^{k+1}\right\}_k$ is convergent [Lemma \ref{DUBD_lemma} \romannumeral 1)], it is satisfied that for any $\delta>0$
	\begin{equation}
		\exists K:\ \Vert z_i^l-z_i^{k+1}\Vert<\delta,\ \ \forall l,k:\  l>k+1\ge K.
	\end{equation}
	Based on (12) and (13), we can infer that
	\begin{equation}\nonumber
		\exists K: \  g_i(z_i^{k+1},\hat{y}_i)\le -\epsilon/2<0,\ \ \forall k:\  k+1\ge K.
	\end{equation}
	Therefore, it can be inferred that for any agent $i\in\mathcal{V}$, there exists a finite number of iterations $K$ in the distributed upper bounding procedure, such that the optimal point obtained by solving the (\ref{dubd}) at the $K$-th iteration satisfies its local semi-infinite constraint in the (\ref{DRCO}).
\end{proof}
	\begin{proposition}
		For any $i\in\mathcal{V}$, take any $\overline Y^{0}_i\subset Y_i$ and proper restriction parameter $\varepsilon^0_i$. Suppose that Assumptions 1-5 hold. Then, the distributed upper bounding procedure is convergent, i.e. ${\sum}^m_{i=1}f_i({z^{k+1}_i})\rightarrow F^*$.
	\end{proposition}
	\begin{proof}
		Our another paper \cite{xunhao} has proved the convergence of the distributed upper bounding procedure, for the sake of completeness a proof is given here. Under Assumption \ref{Assumption 4}, we have $z^{k+1}_i=z^{k+1}$ for all agents $i\in\mathcal{V}$.
		\\Let $a^{k+1}$ and $b^{k+1}$ be the optimal solutions of the optimization problems with $A^k$ and $B^k$ as feasible sets, respectively (see: formula (\ref{8})). It follows that $F(b^{k+1})\le F(z^{k+1})\le F(a^{k+1})$. According to the result of formula (\ref{10}), taking the limit on both sides, we have
		$$\!\mathop{\lim}\limits_{k\rightarrow+\infty}\mathop{\sum}\limits^m_{i=1}\!f_i({z^{k+1}_i})\!=\!\mathop{\lim}\limits_{k\rightarrow+\infty}\!\!F(z^{k+1})\!\ge\! \mathop{\lim}\limits_{k\rightarrow+\infty}\!\!F(b^{k+1})\!=\!F^*\!,$$
		$$\!\mathop{\lim}\limits_{k\rightarrow+\infty}\mathop{\sum}\limits^m_{i=1}\!f_i({z^{k+1}_i})\!=\!\mathop{\lim}\limits_{k\rightarrow+\infty}\!\!F(z^{k+1})\!\le\! \mathop{\lim}\limits_{k\rightarrow+\infty}\!\!F(a^{k+1})\!=\!F^*\!.$$
		Therefore, the sequence $\left\{{\sum}^m_{i=1}f_i({z^{k+1}_i})\right\}_k$ converges to $F^*$.
	\end{proof}
	
	\begin{proposition}
		For all agents $i\in\mathcal{V}$, take any $\overline Y^0_i\subset Y_i$ and proper restriction parameters $\varepsilon^0_i$. Suppose that Assumptions 1-5 hold. Then, for any iteration $k+1> 0$, it is satisfied that ${\sum}^m_{i=1}f_i({\overline x^{k+1}_i})\ge F^*$.
	\end{proposition}
	\begin{proof}
		Based on [Lemma \ref{DUBD_lemma} \romannumeral 2)], the optimal solution $z_i^{k+1}$ of the (\ref{dubd}) for agent $i$ satisfies its semi-infinite constraint in the $(k+1)$-th iteration, i.e. $z_i^{k+1}\in X_i$.
		According to the assumption that the solution of the (\ref{dubd}) for each agent satisfies consensus, we can get $$z_j^{k+1}\in X_i,\ \ \forall j\in\mathcal{V}.$$
		Based on the above formula, we can infer that within a finite number of iterations $K$ of distributed upper bounding procedure, we can obtain a consensus point $z^{K}$ that satisfies $z^{K}\in\bigcap_{i=1}^m X_i$.\\
	    Due to the feasibility of $z^{K}$ for the (\ref{DRCO}), we can conclude that at the K-th iteration, $$\mathop{\sum}\limits^m_{i=1}f_i({\overline x^{K}_i})=\mathop{\sum}\limits^m_{i=1}f_i({z^{K}_i})=\mathop{\sum}\limits^m_{i=1}f_i({ z^{K}})\ge F^*.$$
		In other iterations, it is satisfied that $z^{k+1}\notin\bigcap_{i=1}^m X_i$. For this case, there exists at least one agent $i$ satisfying $f_i(\overline x^{k+1}_i)\rightarrow +\infty$ (see: DUBD Oracle in Section \ref{sec:Distributed Upper Bounding Procedure}). It follows that ${\sum}^m_{i=1}f_i({\overline x^{k+1}_i})\ge F^*$.\\ 
		Overall, for any iteration $k\ge 0$, it is satisfied that ${\sum}^m_{i=1}f_i({\overline x^{k+1}_i})\ge F^*$. 
	\end{proof}
\color{black}
	\subsection{Finite-time Convergence}
	\label{sec:Finite-time Convergence}
	Based on the results from Propositions 1-3, it can be concluded that both the distributed upper bounding procedure and the distributed lower bounding procedure of Algorithm 1 converge to the optimal solution of the (DRCO) while satisfying the following condition: for any $k\ge0$ and $i\in\mathcal{V}$,
	\begin{equation}
		\mathop{\sum}\limits^m_{i=1}f_i({\widetilde x^{k+1}_i})\le F^*\le \mathop{\sum}\limits^m_{i=1}f_i({\overline x^{k+1}_i}).
	\end{equation}
	where $\widetilde x^{k+1}_i$ is the solution of agent $i$ obtained by solving the (\ref{dlbd}), and $\overline x^{k+1}_i$ is the solution output by the DUBD Oracle in the $(k+1)$-th iteration of Algorithm 1.
	Moreover, from [Lemma \ref{DUBD_lemma} \romannumeral 2)], it follows that for a finite number of iterations $K$, we can definitely obtain a set of $\overline x^{K}_i$ satisfying ${\sum}^m_{i=1}f_i({\overline x^{K}_i})<+\infty$. 
	\par
	For a centralized system, as a result of the above results, suppose that we set the formula (\ref{stopping criterion}) as a stopping criterion for Algorithm 1.
	\begin{equation}\label{stopping criterion}
		\vert \mathop{\sum}\limits_{i=1}^{m} f_i(\overline x^{k+1}_i)-\mathop{\sum}\limits_{i=1}^{m}f_i(\widetilde x^{k+1}_i)\vert  \le \epsilon^f,
	\end{equation}
	where $\epsilon^f>0$ is a termination parameter. 
	\par
	When the stopping criterion (\ref{stopping criterion}) is satisfied, we have
	\begin{equation}\label{approximate optimality}
		\begin{aligned}
			&\hspace{-1mm}\vert\mathop{\sum}\limits_{i=1}^{m}\! f_i(\overline x^{k+1}_i)\!-\!F^*\vert\! \le\! \vert \mathop{\sum}\limits_{i=1}^{m}\! f_i(\overline x^{k+1}_i)\!-\!\mathop{\sum}\limits_{i=1}^{m}\!f_i(\widetilde x^{k+1}_i)\vert\! \le\! \epsilon^f\!,\\
			&\vert F^*\!-\!\mathop{\sum}\limits_{i=1}^{m}\!f_i(\widetilde x^{k+1}_i)\vert\! \le\! \vert \mathop{\sum}\limits_{i=1}^{m}\! f_i(\overline x^{k+1}_i)\!-\!\mathop{\sum}\limits_{i=1}^{m}\!f_i(\widetilde x^{k+1}_i)\vert\! \le\! \epsilon^f\!.\\
		\end{aligned}
	\end{equation}
    Let $\hat \epsilon^f$ be the accuracy of the approximate optimal solution of the (\ref{DRCO}) for the centralized case. Hence, we can obtain an $\hat \epsilon^f$-approximate optimal solution of the (\ref{DRCO}) by Algorithm 1, in which $\hat\epsilon^f=\epsilon^f$.
	\par
	However, compared to centralized systems, distributed systems lack a central processing unit that has access to global network information. Instead, each agent can only gather local information about its in-neighbors through agent communication. Consequently, it is not practical to directly sum the local objective function values of all agents to determine the stopping criterion (\ref{stopping criterion}). In order to ensure that Algorithm 1 is finite-time convergent, motivated by the finite-time consensus algorithm \cite{xie2017stop}, this subsection presents two termination methods specifically designed for uniformly strongly connected graphs. Furthermore, the effectiveness of these two methods is compared.
	
	\par 
	{ \bf{\emph {1) Method \uppercase\expandafter{\romannumeral1}:}}}
	The objective of the first method is to make all agents stop updating information when the local objective values obtained by the distributed upper bounding procedure and distributed lower bounding procedure satisfy the following stopping criterion:
	\begin{flalign}\label{stopping_criterion_1}
		\vert f_i(\overline x_i^{k+1})-f_i(\widetilde x_i^{k+1})\vert\le \epsilon^f,\quad\forall i\in\mathcal{V}.
	\end{flalign}
	\par
	To this purpose, our idea is to adapt the minimum-consensus algorithm in \cite{xie2017stop} to propose an internal iterative method of Algorithm 1 executed in a distributed way, so that each agent can track the minimum number that consecutively satisfies the following conditions.
	\begin{flalign}
		\vert f_i(\overline x_i^{k+1})-f_i(\widetilde x_i^{k+1})\vert\le \epsilon^f,\quad\forall i\in\mathcal{N}_i^{in}(t)\cup\left\{i\right\}.
	\end{flalign}
	where the time slot $t\ge 0$ represents the number of internal iterations, and the graph sequence $\left\{\mathcal{G}(t)\right\}$ varies with $t$. 
	\par
	Firstly, each agent sends two-bit data $[h_i(t), c_i(t)]$ to its out-neighbors between the time slots $t$ and $t+1$. 
	Then, each agent calculates (\ref{19})-(\ref{20}) according to its and its in-neighbors' information at time slot $t+1$.
	\begin{flalign}\label{19}
		&	h_i(t+1)=\mathop{\min}\limits_{j\in N^{in}_i(t)\cup\left\{i\right\}}\left\{h_i(t),c_i(t)\right\}+1,&
	\end{flalign}
	\begin{flalign}\label{20}
		&	c_i(t+1)=
		\begin{cases}
			c_i(t)+1,\quad \vert f_j(\overline x_j^{k+1})-f_j(\widetilde x_j^{k+1})\vert \le\epsilon^f, \\ \ \ \qquad\qquad\ \ \forall j\in N^{in}_i(t)\cup\left\{i\right\},\\
			0,\qquad \qquad \quad otherwise,
		\end{cases}&
	\end{flalign}
	where $h_i(0)=0$, $c_i(0)=0$.
	\par
	In the following, we show how to use $h_i$ to check whether the solutions obtained by Algorithm 1 satisfies the stopping criterion (\ref{stopping_criterion_1}).
	\begin{proposition}\label{proposition 1}
		Under Assumption \ref{Assumption 2}, the calculation is performed according to (\ref{19})-(\ref{20}). If at the $[T(m-1)+1]$-th time slot, there is an agent $i\in\mathcal{V}$ satisfying $h_i(t)\ge T(m-1)+1$, the uniformly strongly connected network reaches the stopping criterion of (\ref{stopping_criterion_1}). 
	\end{proposition}
	\begin{proof}
		We refer to the proof idea in \cite{xie2017stop} to prove this proposition.
		From Assumption \ref{Assumption 2} that $\mathcal G(t:t+T)$ is strongly connected, there exists a directed path $(i,i_1),(i_1,i_2),...,(i_d,j)$ from $i$ to $j$ with $d\le T(m-1)-1$ for any $i\neq j\in\mathcal{V}$. \\
		Suppose that there is an agent $j\in\mathcal{V}$ satisfying $h_j(t)\ge T(m-1)+1$ at time slot t, then at time slot $(t-1)$, it follows that 
		\begin{flalign}\nonumber
			h_{i_{d}}(t-1)\ge T(m-1), \\ \nonumber
			c_{i_{d}}(t-1)\ge T(m-1). \nonumber
		\end{flalign}
		Similarly, when the time slot is $(t-2)$, it follows that 
		\begin{flalign}\nonumber
			h_{i_{d-1}}(t-2)\ge T(m-1)-1, \\ \nonumber
			c_{i_{d-1}}(t-2)\ge T(m-1)-1. \nonumber
	\end{flalign}
	Repeat the same steps, it follows that
	\begin{flalign}\nonumber
		h_{i}(t-d-1)\ge T(m-1)-d\ge 1, \\ \nonumber
		c_{i}(t-d-1)\ge T(m-1)-d\ge 1. \nonumber
	\end{flalign}
	Therefore, the stopping criterion (\ref{stopping_criterion_1}) is reached at time slot $(t-d-1)$.
\end{proof}
\par
Assuming that there is an agent $i$ satisfying $h_i\ge T(m-1)+1$, this agent will issue an exit command to stop updating its information. Since the values of $\widetilde x^{k+1}_i$ and $\overline x^{k+1}_i$ do not change with time slot $t$ in this termination method, it is obvious that all other agents also satisfy $h_j\ge T(m-1)+1,\  i\neq j\in\mathcal{V}$. Overall, all agents can stop updating information simultaneously by their own exit commands, in contrast to the Literature \cite{xie2017stop} where they cannot terminate simultaneously and need to broadcast the exit command to their out-neighbors.
\begin{theorem}
	Under Assumptions 1-5, Algorithm 1 terminates finitely and generates a feasible $\widetilde \epsilon^f$-approximate optimal consensus solution of the (\ref{DRCO}), where the accuracy of approximate optimality is $\widetilde \epsilon^f=m\epsilon^f$. This holds for any reduciton parameter $r\ge 1$, any finite sets $\widetilde Y_i^0\subset Y_i$, $\overline Y_i^0\subset Y_i$, and proper restriction parameter $\varepsilon_i^0>0$.
\end{theorem}
\begin{proof}
	From Propositions 1-4, it is straightforward to derive the finite-time convergence of Algorithm 1. 
	Assume that Algorithm 1 terminates at the $k1$-th iteration. During the $k1$-th iteration, the $\overline x^{k1}_i$ generated by DUBD Oracle satisfies the following conditions: $$x^{k1}_i\in X_i,\ \forall i\in\mathcal{V},$$ and it is assigned to $x^{opt}_i$. Under Assumption 5, we can conclude that the solution obtained by Algorithm 1 satisfies the feasibility of the $(\ref{DRCO})$.
	\par
	Therefore, Algorithm 1 converges to a feasible consensus solution in a finite number of iterations. Next, we demonstrate the approximate accuracy of the solution obtained by Algorithm 1 to the optimal solution of the (\ref{DRCO}).
	\\
	Due to the property of triangle inequality, it follows that
	\begin{flalign}\nonumber
		\begin{aligned}
			\!\vert\! \mathop{\sum}\limits_{i=1}^{m}\! f_i(\overline x^{k+1}_i)\!-\!\mathop{\sum}\limits_{i=1}^{m}\!f_i(\widetilde x^{k+1}_i)\!\vert\! &=\vert\! \mathop{\sum}\limits_{i=1}^{m} [f_i(\overline x^{k+1}_i)\!-\!f_i(\widetilde x^{k+1}_i)]\vert\! \\&\le\mathop{\sum}\limits_{i=1}^{m} \!\vert f_i(\overline x^{k+1}_i)\!-\!f_i(\widetilde x^{k+1}_i)\vert\! \\
			&\le m \epsilon^f=\widetilde \epsilon^f.
		\end{aligned}
	\end{flalign}	
	Combining the inequality relations of (\ref{approximate optimality}), the accuracy of the approximate optimal solution obtained by Algorithm 1 is $\widetilde \epsilon^f=m\epsilon^f$.
\end{proof}
\par
Note that the accuracy of the approximate optimal solution is only related to the number of agents $m$ and the value of the termination parameter. There is a drawback to this termination method: when the number of agents in the multi-agent system is very large, an extremely small value of the termination parameter is required to guarantee the accuracy of the approximate optimal solution, which imposes high demands on the computational accuracy during the numerical solution, and greatly increases the number of iterations $k$ of Algorithm 1.
\par 
{ \bf{\emph {2) Method \uppercase\expandafter{\romannumeral2}:}}}
To address this issue, we make an improvement to the aforementioned method. The following is the stopping criterion:
\begin{flalign}\label{stopping_criterion_2}
	\mathop{\sum}\limits_{j\in N^{in}_i(t)\cup\left\{i\right\}}\vert f_j(\overline x_j^{k+1})-f_j(\widetilde x_j^{k+1})\vert\le \epsilon^f,\quad \forall i\in\mathcal{V}.
\end{flalign}
\par
 Our primary idea is to compute the value of $e_i=\vert f_i(\overline x_i^{k+1})-f_i(\widetilde x_i^{k+1})\vert$ for each agent and to transmit the value $e_i$ to its out-neighbors after the calculation of the distributed lower bounding procedure and distributed upper bounding procedure of Algorithm 1. Then, we adapt the Method \uppercase\expandafter{\romannumeral1} so that each agent can keep track of the minimum number that consecutively satisfies the following conditions:
\begin{flalign}
	\mathop{\sum}\limits_{j\in N^{in}_i(t)\cup\left\{i\right\}} \vert f_j(\overline x_j^{k+1})-f_j(\widetilde x_j^{k+1})\vert\le \epsilon^f.
\end{flalign}
\par
The adapted termination method is shown below:
firstly, each agent sends two-bit data $[h_i(t), c_i(t)]$ to its out-neighbors between the time slots t and t+1.
Then, each agent calculates (\ref{23})-(\ref{24}) according to its own and in-neighbors' information at time slot $t+1$.
\begin{flalign}\label{23}
	&	h_i(t+1)=\mathop{\min}\limits_{j\in N^{in}_i(t)\cup\left\{i\right\}}\left\{h_i(t),c_i(t)\right\}+1,&
\end{flalign}
\begin{flalign}\label{24}
	&	c_i(t+1)=
	\begin{cases}
		c_i(t)+1,\quad \mathop{\sum}\limits_{j\in N^{in}_i(t)\cup\left\{i\right\}} e_j\le \epsilon^f,\\
		0,\qquad \qquad \quad otherwise,
	\end{cases}&
\end{flalign}
where $h_i(0)=0$, $c_i(0)=0$, $e_j=\vert f_j(\overline x_j^{k+1})-f_j(\widetilde x_j^{k+1})\vert$.
\par
Similar to the Method \uppercase\expandafter{\romannumeral1}, we show how to check whether the solutions of Algorithm 1 satisfy the stopping criterion (\ref{stopping_criterion_2}) by using $h_i$.
\begin{proposition}\label{proposition 5}
	Under Assumption \ref{Assumption 2}, the calculation is performed according to (\ref{23})-(\ref{24}).  If at the $[T(m-1)+1]$-th time slot, there is an agent $i\in\mathcal{V}$ satisfying $h_i(t)\ge T(m-1)+1$, the network reaches the stopping criterion of (\ref{stopping_criterion_2}).
\end{proposition}
\begin{proof}
	Similar to the proof of Proposition 4, the above result is straightforward.
\end{proof}
\begin{theorem}
	Under Assumptions 1-5, suppose that the graph sequence $\mathcal{G}(t)$ is known. For any reduciton parameter $r\ge 1$, any finite sets $\widetilde Y_i^0\subset Y_i$, $\overline Y_i^0\subset Y_i$, and proper restriction parameter $\varepsilon_i^0>0$, Algorithm 1 terminates finitely and generates a feasible $\overline \epsilon^f$-approximate optimal consensus solution of the (\ref{DRCO}), where the accuracy of approximate optimality $\overline \epsilon^f$ is an optimal objective of the following linear program:
	\begin{flalign} \label{linear program}
		\begin{aligned} 
			\overline \epsilon^f=&{\max} &  & \mathop{\sum}\limits^m_{i=1}e_i\\
			&\rm{s.t.} &&\mathop{\sum}\limits^m_{i=1}\mathop{\sum}\limits^T_{t=1}\mathop{\sum}\limits_{j\in N^{in}(t)_i\cup\left\{i\right\}}e_j \le m T\cdot \epsilon^f,\\
			&&& 0\le e_i\le \epsilon^f,\quad i=1,...,m.
		\end{aligned}
	\end{flalign}
\end{theorem}

\begin{proof}
	Similar to the proof of Theorem 1, we can easily establish that Algorithm 1 terminates within a finite number of iterations and yields a feasible approximate optimal consensus solution of the (\ref{DRCO}) based on Propositions 1-3 and 5. Next, we provide a proof of the accuracy of the approximate optimal solution. \\When Algorithm 1 has reached the stopping criterion (\ref{stopping_criterion_2}), there is
	\begin{equation}\label{naive_constraint}
		0\le e_i\le \mathop{\sum}\limits_{j\in N^{in}_i(t)\cup\left\{i\right\}} e_j\le \epsilon^f, \ \forall i\in\mathcal{V},\ \forall 1\le t\le T,
	\end{equation}	
	where the graph $\mathcal{G}(1:T)$ is strongly connected.\\
	By relaxing the constraint (\ref{naive_constraint}), we can obtain the constraints shown in (\ref{linear program}).\\
	Therefore, the accuracy of the approximate optimal solution obtained by Algorithm 1 is the maximum value of $\sum_{i=1}^m e_i$ subject to the relaxed constraints, as shown in (\ref{linear program}). 
\end{proof}
\begin{figure}[!t]
	\centerline{\includegraphics[width=\columnwidth]{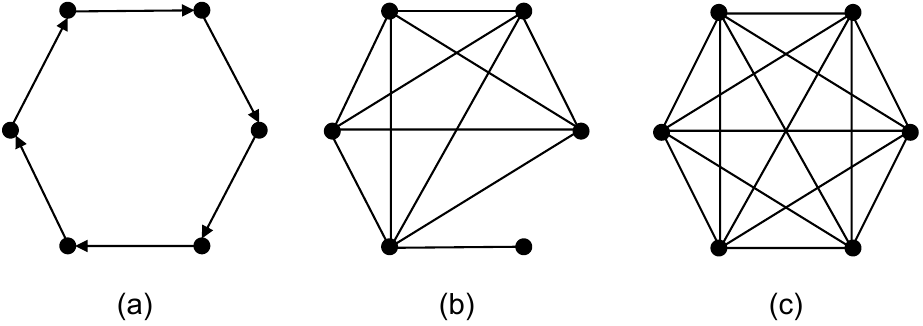}}
	\caption{Three types of graphs. (a) Directed cycle graph. (b) Customized graph. (c) Complete graph.}
	\label{fig2}
	\centerline{\includegraphics[width=\columnwidth]{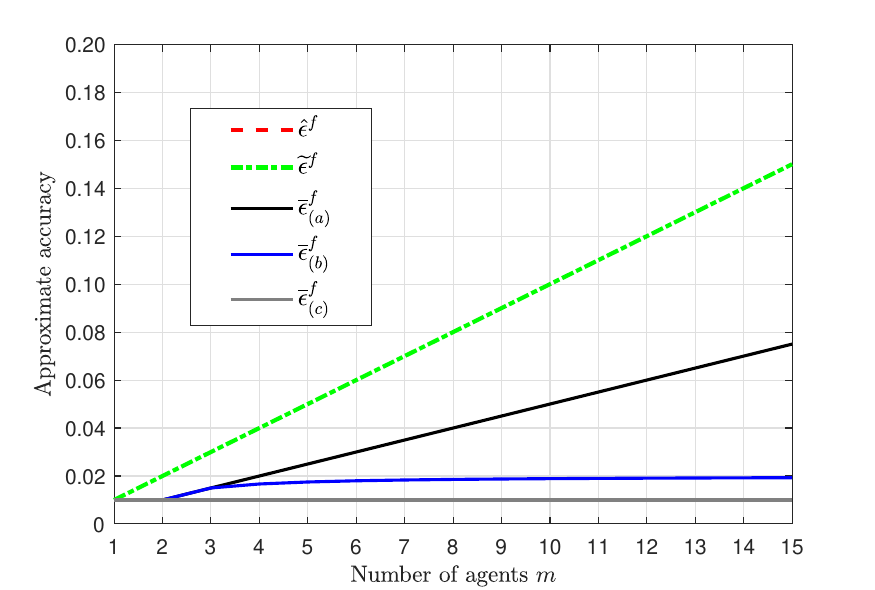}}
	\caption{Accuracy of the approximate optimal solution obtained by Algorithm 1 as a function of the number of agents and the graph structures, for the case where the termination parameter $\epsilon^f=0.01$. The red dashed line indicates the case of centralized systems. The green dotted-dashed line represents the case of the Method \uppercase\expandafter{\romannumeral1}. The solid line refers to the case of the Method \uppercase\expandafter{\romannumeral2}, where the black solid line, blue solid line and grey solid line correspond to the network structures of the directed cycle graphs, customized graphs, and complete graphs, respectively.}
	\label{fig3}
\end{figure}
\par
Based on the results of Theorems 1 and 2, it can be established that Algorithm 1 can finite-time converge to a feasible consensus solution satisfying global optimality to a certain accuracy of the (\ref{DRCO}). Furthermore, the result of Theorem 2 is significantly less conservative than that of Theorem 1 since the optimal objective of the linear program (25) satisfies $\overline \epsilon^f\le m\epsilon^f$, whereas this $\overline\epsilon^f$ is related to network graphs. This can be observed through a numerical example in Fig. 3, where we study how $\hat \epsilon^f$, $\widetilde \epsilon^f$ and $\overline \epsilon^f$ change as a function of the number of agents $m$ in three types of graphs. Algorithm 1 is designed with a termination parameter $\epsilon^f=0.01$. The approximate accuracy $\hat\epsilon^f$ (represented by the red dashed line) in centralized systems is independent of the number of agents $m$ and graphs, and therefore remains constant as $m$ increases. In the case of the Method \uppercase\expandafter{\romannumeral1}, the approximate accuracy $\widetilde\epsilon^f$ (represented by the green dotted-dashed line) linearly increases with $m$ as $\widetilde\epsilon^f=m\epsilon^f$, but independent of graphs. However, there is a different pattern in the case of the Method \uppercase\expandafter{\romannumeral2}, influenced by both the number of agents and the network structures. For directed cyclic graphs, the approximate accuracy $\overline\epsilon^f_{(a)}$ (represented by the black solid line) increases linearly with the number of agents at $m\ge 2$, where the rate of growth is less than that of the Method \uppercase\expandafter{\romannumeral1}. The accuracy of the approximate optimal solution for complete graphs can attain levels comparable to those achieved by the centralized system, i.e., $\overline\epsilon^f_{(c)}=\epsilon^f$ (see the grey solid line). For customized graphs, the approximate accuracy $\overline\epsilon^f_{(b)}$ (represented by the blue solid lines) increases moderately with $m$. Note that here for customized graphs, we consider a specific example of a time-invariant strongly connected network composed of $m$ agents. Among these $m$ agents, $(m-1)$ agents are fully connected to each other, while the remaining one agent is only connected to the $(m-1)$-th agent. Overall, This Method \uppercase\expandafter{\romannumeral2} provides a less conservative result compared to the Method \uppercase\expandafter{\romannumeral1} while still allowing for distributed information, which contrasts the centralized systems. 
\section{Numerical Case Studies}
In order to verify the effectiveness of the proposed Algorithm \ref{alg:1}, we consider a distributed robust convex optimization problem with bounded uncertainty, as shown below.
\begin{flalign}
	\begin{aligned}
		&\mathop{\min}\limits_{x\in \mathcal{F}}  F(x)=\sum^m_{i=1}\Vert x-u_i\Vert^2\\
		&\ {\rm{s.t.}}\ g_i(x,y_i)=(x(1)-v_i)^2+2y_ix(2)-{y_i}^2-1\le 0,\\
		&\qquad \qquad \qquad \qquad \qquad \ \ \forall y_i\in[-1,1],\quad i=1,...,m,
	\end{aligned}
\end{flalign}
where the decision variable $x=[x(1),x(2)]^\top$, the global constraint $\mathcal{F}\!=\!\left\{\!x\in\mathbb{R}^2|\!-2\!\le\! x(1)\!\le\!2, \!-1\!\le\!  x(2)\!\le\!1\right\}\!$. $u_i\in \mathbb{R}^2$ and $v_i\in \mathbb{R}$ are the local objective function vector and local constraint parameter for agent $i$, respectively, where the corresponding values are shown in Table \ref{Table1}. Furthermore, $y_i\in [-1,1]$ is the uncertain parameter for agent $i$.
	\begin{table}[!h]\large
		\caption{Parameter values of the problem}
		\label{Table1}
		\centering 
		\resizebox{\linewidth}{!}{
		\begin{tabular}{ccccccc} 
			\toprule 
			Agent & Agent 1 & Agent 2 & Agent 3 &
			Agent 4 & Agent 5 & Agent 6  \\ 
			\midrule 
			$u_i$ & $[0,6]$ & $[0,0]$ & $[1,1]$& $[-1,-1]$& $[1,-1]$& $[-1,1]$ \\  
			$v_i$ & $-0.75$ & $-0.5$ & $-0.25$ & $0.25$ & $0.5$ & $0.75$ \\		
			\bottomrule 
		\end{tabular} 
	}
	\end{table}
\par
We adopt three types of network graphs, see Fig. 2, to verify that the distributed robust convex optimization algorithm (see: Algorithm \ref{alg:1}) terminates in a finite number of iterations and to illustrate that the solutions of all agents are feasible with respect to their own local constraints in (\ref{DRCO}). The implementation is carried out in MATLAB Version 9.5.0.944444 (R2018b, win64) and runs on an Intel(R) Core (TM) i7-7700HQ CPU @ 2.80GHz, 256GB terminal server. In addition, this section compares Algorithm 1 with the existing related algorithms.
\subsection{Effectiveness of the distributed robust convex optimization algorithm}
\par
We initialize the parameters of Algorithm \ref{alg:1} as follows: the initial restriction parameters $\varepsilon_i^0 = 0.01$, reduction parameter $r=2$, two uncertainty sets $\widetilde Y_i^0 =\emptyset$, $\overline Y_i^0 =\emptyset$, and termination parameter $\epsilon^f = 0.01$. It is worth noting that for the initial values of the restriction parameters $\varepsilon_i^0$, we ensure that there exists at least one point $\hat x=[0,0]^\top$ satisfying $$g_i(\hat x,y_i)+\varepsilon^{0}_i\le0,\ \forall y_i\in Y_i, \quad i=1,...,m.$$
Therefore, the solvability of the (\ref{dubd}) is guaranteed for each iteration of Algorithm \ref{alg:1}.
\par
Then, we implement Algorithm 1 to solve the above numerical case. In solving the two problems ((\ref{dlbd}) and (\ref{dubd})), we use a strategy by combining the distributed random-fixed projection algorithm \cite{xie2018distributed} and the finite-time consensus algorithm \cite{xie2017stop} to obtain the optimal consensus solutions of the (\ref{dlbd}) (or the (\ref{dubd})) within a finite number of time slots. Note that in order to ensure the consensus of the solutions obtained by solving the (\ref{dlbd}) (or the (\ref{dubd})) for all agents, we set a pretty high consensus accuracy $10^{-4}$ on the finite-time consensus algorithm. In addition, since the local constraint functions of the agents are concave with respect to the uncertain parameters $y_i$, the solutions of the LLPs are rigorously solved by the analytical method.
\begin{figure}[!h]
	\centerline{\includegraphics[width=\columnwidth]{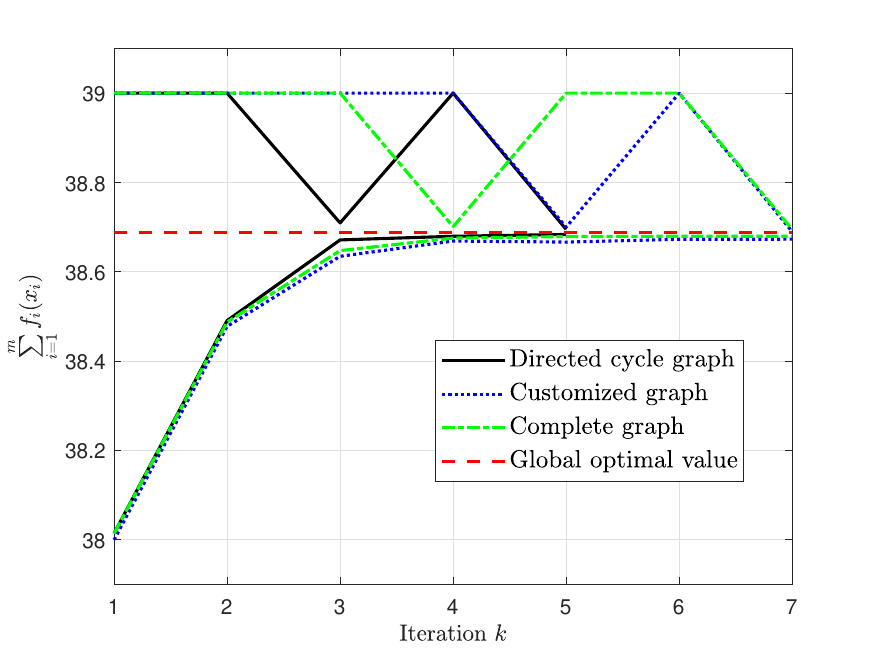}}
	\centering (a)
	\centerline{\includegraphics[width=\columnwidth]{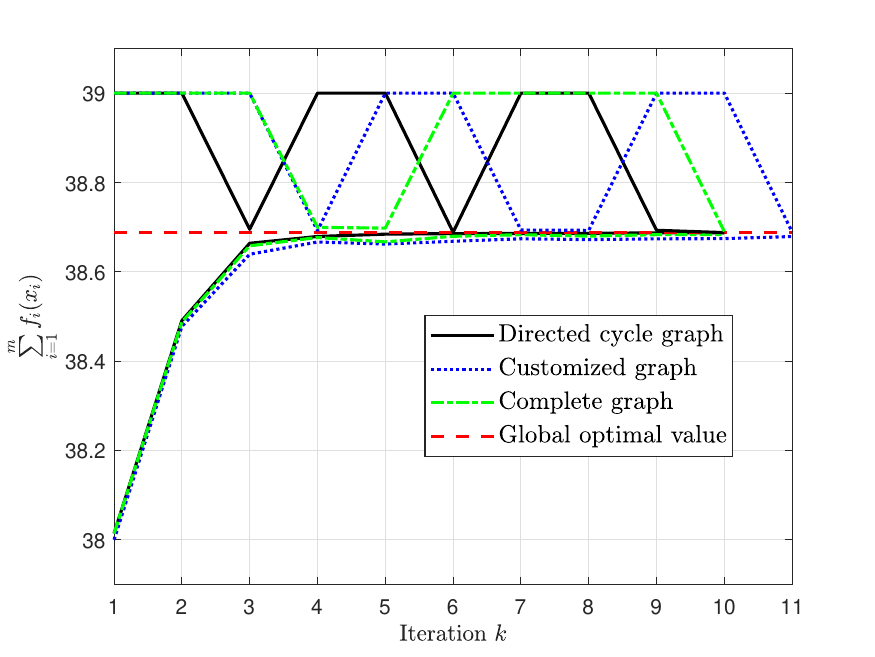}}
	\centering (b)
	\caption{Convergence process of Algorithm 1 on three types of network graphs. (a) Method \uppercase\expandafter{\romannumeral1} as the termination method. (b) Method \uppercase\expandafter{\romannumeral2} as the termination method. The red dashed lines indicate the global optimal value of the numerical case. The black solid lines represent the convergence process of Algorithm 1 under a directed cycle graph. The blue dotted lines indicate the convergence process of Algorithm 1 on the customized graph. The green dotted-dashed lines show the convergence process of Algorithm 1 over a complete graph.}
	\label{fig4}
\end{figure}
	\begin{table*}[!htb]
	\caption{Numerical Results of the Simulation}
	\label{Table2}
	\centering 
	\resizebox{1.0\linewidth}{!}
	{
		\begin{tabular}{|c|c|c|c|c|c|c|c|c|} 
			\hline
			\multicolumn{3}{|c|}{\diagbox{{Content}}{Agent}} & Agent 1 & Agent 2 & Agent 3 &
			Agent 4 & Agent 5 & Agent 6  \\ 
			\hline 
			\multirow{9}*{\makecell{Naive \\Method}} &\multirow{3}*{\makecell{Directed\\ Cycle Graph}} & Optimal point &[-0.0000;0.6597]&[+0.0001;0.6597]&[+0.0000;0.6598]&[+0.0001;0.6597]&[+0.0001;0.6597]&[+0.0001;0.6598] \\
			\cline{3-9}
			~& ~&Local feasibility &$\checkmark$&$\checkmark$&$\checkmark$&$\checkmark$&$\checkmark$ &$\checkmark$\\
			\cline{3-9}
			~& ~&Objective value &\multicolumn{3}{|c|}{$\sum_{i=1}^m f_i(\widetilde x_i^{k+1})=38.6835$} &\multicolumn{3}{|c|}{$\sum_{i=1}^m f_i(\overline x_i^{k+1})=38.6948$}\\
			\cline{2-9}
			~&\multirow{3}*{\makecell{Random\\ Graph}} & Solution &[-0.0000;0.6610]&[-0.0000;0.6610]&[-0.0000;0.6611]&[+0.0001;0.6611]&[+0.0001;0.6610]&[+0.0000;0.6611] \\
			\cline{3-9}
			~& ~&Local feasibility &$\checkmark$&$\checkmark$&$\checkmark$&$\checkmark$&$\checkmark$ &$\checkmark$\\
			\cline{3-9}
			~& ~&Objective value &\multicolumn{3}{|c|}{$\sum_{i=1}^m f_i(\widetilde x_i^{k+1})=38.6729$} &\multicolumn{3}{|c|}{$\sum_{i=1}^m f_i(\overline x_i^{k+1})=38.6896$}\\
			\cline{2-9}
			~&\multirow{3}*{\makecell{Complete\\ Graph}} & Solution &[-0.0001;0.6598]&[-0.0000;0.6598]&[-0.0001;0.6598]&[-0.0001;0.6599]&[+0.0000;0.6599]&[+0.0000;0.6598] \\
			\cline{3-9}
			~& ~&Local feasibility &$\checkmark$&$\checkmark$&$\checkmark$&$\checkmark$&$\checkmark$ &$\checkmark$\\
			\cline{3-9}
			~& ~&Objective value &\multicolumn{3}{|c|}{$\sum_{i=1}^m f_i(\widetilde x_i^{k+1})=38.6794$} &\multicolumn{3}{|c|}{$\sum_{i=1}^m f_i(\overline x_i^{k+1})=38.6946$}\\
			\hline
			\multirow{9}*{\makecell{Tighter \\Method}} &\multirow{3}*{\makecell{Directed\\ Cycle Graph}} & Solution &[-0.0000;0.6611]&[-0.0000;0.6612]&[+0.0000;0.6612]&[-0.0001;0.6612]&[+0.0000;0.6611]&[+0.0000;0.6612] \\
			\cline{3-9}
			~& ~&Local feasibility &$\checkmark$&$\checkmark$&$\checkmark$&$\checkmark$&$\checkmark$ &$\checkmark$\\
			\cline{3-9}
			~& ~&Objective value &\multicolumn{3}{|c|}{$\sum_{i=1}^m f_i(\widetilde x_i^{k+1})=38.6773$} &\multicolumn{3}{|c|}{$\sum_{i=1}^m f_i(\overline x_i^{k+1})=38.6892$}\\
			\cline{2-9}
			~&\multirow{3}*{\makecell{Random\\ Graph}} & Solution &[-0.0000;0.6612]&[-0.0000;0.6612]&[-0.0000;0.6611]&[-0.0000;0.6612]&[-0.0001;0.6611]&[+0.0000;0.6612] \\
			\cline{3-9}
			~& ~&Local feasibility &$\checkmark$&$\checkmark$&$\checkmark$&$\checkmark$&$\checkmark$ &$\checkmark$\\
			\cline{3-9}
			~& ~&Objective value & \multicolumn{3}{|c|}{$\sum_{i=1}^m f_i(\widetilde x_i^{k+1})=38.6788$} & \multicolumn{3}{|c|}{$\sum_{i=1}^m f_i(\overline x_i^{k+1})=38.6887$}\\
			\cline{2-9}
			~&\multirow{3}*{\makecell{Complete\\ Graph}} & Solution &[-0.0000;0.6610]&[+0.0000;0.6610]&[-0.0001;0.6611]&[-0.0000;0.6610]&[-0.0001;0.6611]&[+0.0000;0.6611] \\
			\cline{3-9}
			~& ~&Local feasibility &$\checkmark$ & $\checkmark$ & $\checkmark$&$\checkmark$& $\checkmark$ &$\checkmark$\\
			\cline{3-9}
			~& ~&Objective value & \multicolumn{3}{|c|}{$\sum_{i=1}^m f_i(\widetilde x_i^{k+1})=38.6832$} & \multicolumn{3}{|c|}{$\sum_{i=1}^m f_i(\overline x_i^{k+1})=38.6896$}\\
			\hline
		\end{tabular} 
}	\end{table*}
\par
The graphical results of Algorithm \ref{alg:1} are shown in Fig. \ref{fig4}, where Fig. \ref{fig4} (a) presents the convergence process of Algorithm \ref{alg:1} with Method \uppercase\expandafter{\romannumeral1} as the termination method, while Fig. \ref{fig4} (b) illustrates the convergence process of Algorithm \ref{alg:1} employing Method \uppercase\expandafter{\romannumeral2} as the termination method. Furthermore, we considered three types of network graphs, where the black solid lines indicate the convergence process of Algorithm \ref{alg:1} under a directed cycle graph, the blue dotted lines represent the convergence process of Algorithm \ref{alg:1} over the customized graph, and the green dotted-dashed lines refer to the case of a complete graph. Note that to better present the convergence process of Algorithm \ref{alg:1}, for the case of $\sum_{i=1}^m f_i(\overline x^k_i)=\infty$, we assign a value of $39$ to $\sum_{i=1}^m f_i(\overline x^k_i)$ in this iteration. Based on the Fig. \ref{fig4}, we can draw the following results:
\begin{enumerate}
	\item Algorithm \ref{alg:1} converges to the global optimal value within a finite number of iterations.
	\item The solutions obtained from the distributed upper and lower bounding procedures serve as upper and lower bounds for the global optimal value, respectively.
	\item Compared to Method \uppercase\expandafter{\romannumeral1}, using Method \uppercase\expandafter{\romannumeral2} as the termination method leads to a higher number of iterations for Algorithm 1 to terminate. This further indicates that the algorithm with Method \uppercase\expandafter{\romannumeral2} as the termination method can obtain a higher accuracy solution than Method \uppercase\expandafter{\romannumeral1}.
\end{enumerate}
In addition, the effect of different network graphs on the number of iterations of Algorithm 1 is considered. For Algorithm 1 with Method I as the termination method, the number of iterations is independent of the network graphs. Conversely, when adopting Method II as the termination method in Algorithm 1, the number of iterations is related to the network graphs. The main reason is that, for identical termination parameter $\epsilon^f$ and the number of agents $m$, the approximation accuracy of the solutions obtained by Algorithm 1 with Method II as the termination method is different in terms of different network graphs (see: Fig. \ref{fig3}). The complete graph corresponds to the highest approximation accuracy and requires Algorithm 1 to perform a greater number of iterations. On the contrary, the directed cycle graph has the lowest approximation accuracy, resulting in the fewest iterations of Algorithm 1. However, since the number of agents considered in this numerical case is small, the effect of the network graph on the number of iterations of Algorithm 1 is not significant.
\par
The numerical results of Algorithm \ref{alg:1} are illustrated in Table \ref{Table2}, where all the data provided have been rounded to four decimal places. The global optimal solution of the numerical case is $x^*=[0,\sqrt{7}/4]^\top$, and the corresponding global optimal value is $F^*\approx 38.687746$. According to Table \ref{Table2}, we can conclude that each agent can obtain a locally feasible solution satisfying global optimality to a certain accuracy when Algorithm 1 terminates finitely. 
\subsection{Comparison with Relevant Algorithms}
\begin{figure*}[!t]
	\centerline{\includegraphics[width=1.6\columnwidth]{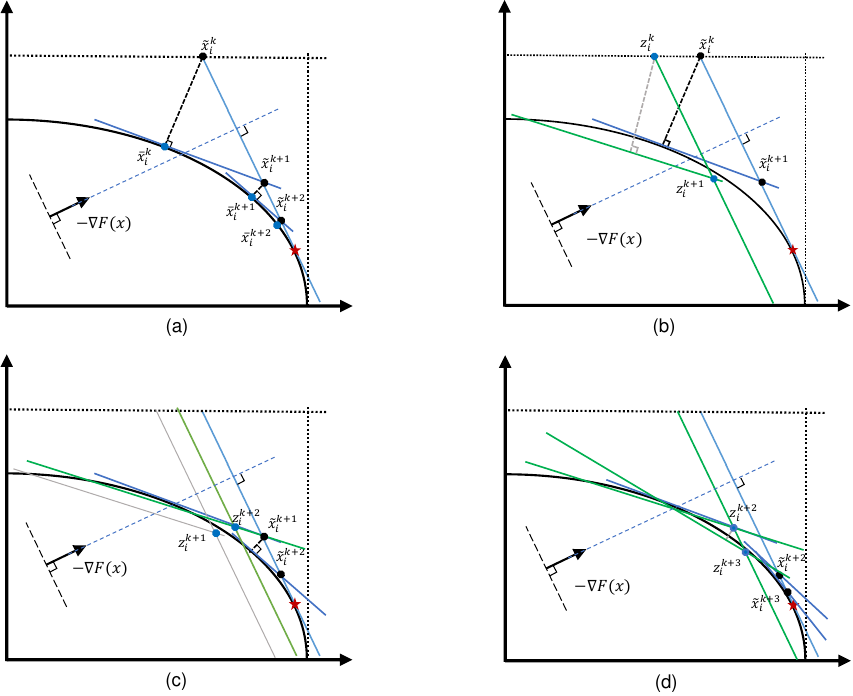}}
	\caption{Graphic illustration of the different strategies of the proposed distributed robust convex optimization algorithm and some relevant algorithms in \cite{burger2013polyhedral,burger2012distributed,yang2014distributed,xunhao}. Assume that the area enclosed by the arc and the coordinate axes is the feasible region of a problem, the direction marked by the black arrow is the gradient descent direction of the objective function, and the red star represents the optimal solution of the problem. The iterative update of the black points in Figure (a) shows the iterative process of the distributed cutting-plane consensus algorithm in \cite{burger2013polyhedral,burger2012distributed}. The iterative update of the black and blue points in Figure (a) indicates the iterative process of the distributed cutting-plane primal-dual algorithm in \cite{yang2014distributed}. The iterative update of the blue points in Figures (b), (c), and (d) presents the iterative update process of the distributed cutting-surface consensus algorithm in \cite{xunhao}, and the iterative update of the black and blue points in Figures (b), (c), and (d) show the iterative update process of Algorithm \ref{alg:1} in this article.}
		\label{fig5}
	\end{figure*}
This subsection compares Algorithm \ref{alg:1} with some related algorithms. To clearly illustrate the differences between Algorithm \ref{alg:1} and the related algorithms, we consider the simplest case: a system composed of only one agent. Our goal is to find the minimum value of the objective function under the set constraints (the area enclosed by the two dotted lines and coordinate axes) and semi-infinite constraints (the area below the arc). Fig. \ref{fig5} graphically illustrates different strategies for solving this (\ref{DRCO}). Note: both Algorithm \ref{alg:1} and the related algorithms construct approximation problems for (\ref{DRCO}) at each iteration. In the case that there are an infinite number of solutions to the approximation problem, we select the one with the smallest value of the horizontal coordinate as the optimal solution for the approximation problem.
\par
The distributed cutting-plane consensus algorithm in \cite{burger2012distributed,burger2013polyhedral} is based on iteratively approximating the (\ref{DRCO}) by successively populating the cutting-planes into the existing finite set of constraints (see: blue solid lines in Fig. \ref{fig5} (a)). The black points in Fig. \ref{fig5} (a) indicate the asymptotic convergence process of the agent. However, the agent cannot obtain a solution with guaranteed feasibility within a finite number of iterations.
\par
The distributed cutting-surface consensus algorithm in \cite{xunhao} is based on iteratively approximating the (\ref{DRCO}) by successively reducing the restriction parameters of the right-hand constraints and populating the cutting-surfaces into the existing finite set of constraints. As shown in Fig. \ref{fig5} (b), (c), and (d), the iterative update of the blue points indicates the iterative process of this algorithm. When the decision variable of the agent lies outside the feasible domain, a cutting-surface is populated into the existing finite set of constraints (see: the green lines in Fig. \ref{fig5} (b) and (d)).
Conversely, when the decision variable lies within the feasible domain, the original constraints in the constraint set (see: the grey lines in Fig. 5 (c)) move toward the feasible domain boundary (see: the green lines in Fig. 5 (c)) due to the reduction of the restriction parameters. This algorithm asymptotically converges to the optimal solution of the (\ref{DRCO}), and the agent can obtain a solution that meets the feasibility in a finite number of iterations.
\par
However, the above algorithms cannot locate the global optimal solution for the certain accuracy in a finite number of iterations. Literature \cite{yang2014distributed} presents a distributed cutting-plane primal-dual algorithm by adding projection operation to the cutting-plane consensus algorithm, which converges to the optimal solution of the (\ref{DRCO}) from the outer and inner directions of the feasible region (see: the black and blue points in Fig. \ref{fig5} (a)). This algorithm, together with the distributed termination method in our article, can get the global optimal solution with a certain accuracy. Nonetheless, this method is confined to the (\ref{DRCO}) with special constraint structures \cite{yang2014distributed}. 
\par
The algorithm in this article combines the advantages of the above algorithms so that each agent can obtain a feasible consensus solution satisfying global optimality to a certain accuracy of the (\ref{DRCO}) within a finite number of iterations. The iterative convergence process of Algorithm \ref{alg:1} is shown in the black and blue points in Figures (b), (c), and (d).
\section{Conclusions and future work}
Based on the right-hand restriction approach proposed in \cite{Mitsos}, a distributed robust convex optimization algorithm is proposed for locating a feasible solution for each agent satisfying global optimality to a certain accuracy of the DRCO within a finite number of iterations. In addition, two distributed termination algorithms, namely the Method \uppercase\expandafter{\romannumeral1} and the Method \uppercase\expandafter{\romannumeral2}, are proposed, which ensure all the agents terminate simultaneously when approximate optimal solutions for a certain accuracy are obtained. The Method \uppercase\expandafter{\romannumeral2} is less conservative than the Method \uppercase\expandafter{\romannumeral1} in terms of the accuracy guarantee of the global optimality, but the accuracy of the solution obtained by the Method \uppercase\expandafter{\romannumeral2} is related to the network structures.
\par
Direct extensions may lie in the following two aspects. In the proposed algorithm, finite-time convergence is proved. We could consider analyzing the convergence rate of the algorithm by thorough theoretical and computational analysis. In addition, this paper mainly considers the DRCO, where the global cost function is strictly convex, and the decision variables are all continuous. Future research would consider extending the proposed algorithm to the case where the local cost function is nonconvex, or some of the decision variables are constrained to integer values.

\bibliographystyle{plain}

\par\noindent
\parbox[t]{\linewidth}{
	\noindent\parpic{\includegraphics[height=1.2in,width=1in,clip,keepaspectratio]{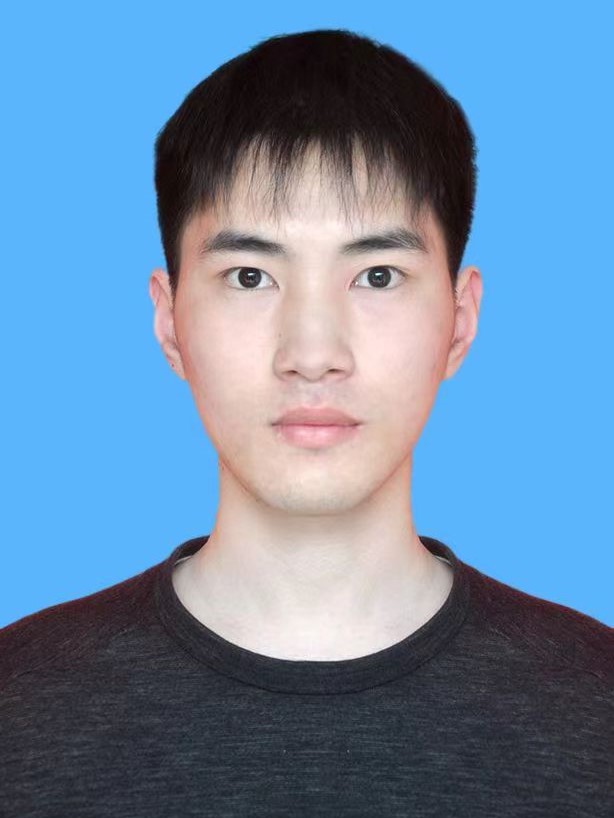}}
	\noindent {\bf Xunhao Wu}\
	received the B.S. degree in automation from Northeastern University, Shenyang, China, in 2021, where he is currently pursuing the Ph.D. degree in control theory and control engineering with the State Key Laboratory of Synthetical Automation for Process Industries, Northeastern University, Shenyang, China.
	His current research interests cover stochastic optimization, robust optimization and their applications in multi-agent systems.}
\vspace{4\baselineskip}                                        
\vspace{-12mm}
\par\noindent
\parbox[t]{\linewidth}{
	\noindent\parpic{\includegraphics[height=1.2in,width=1in,clip,keepaspectratio]{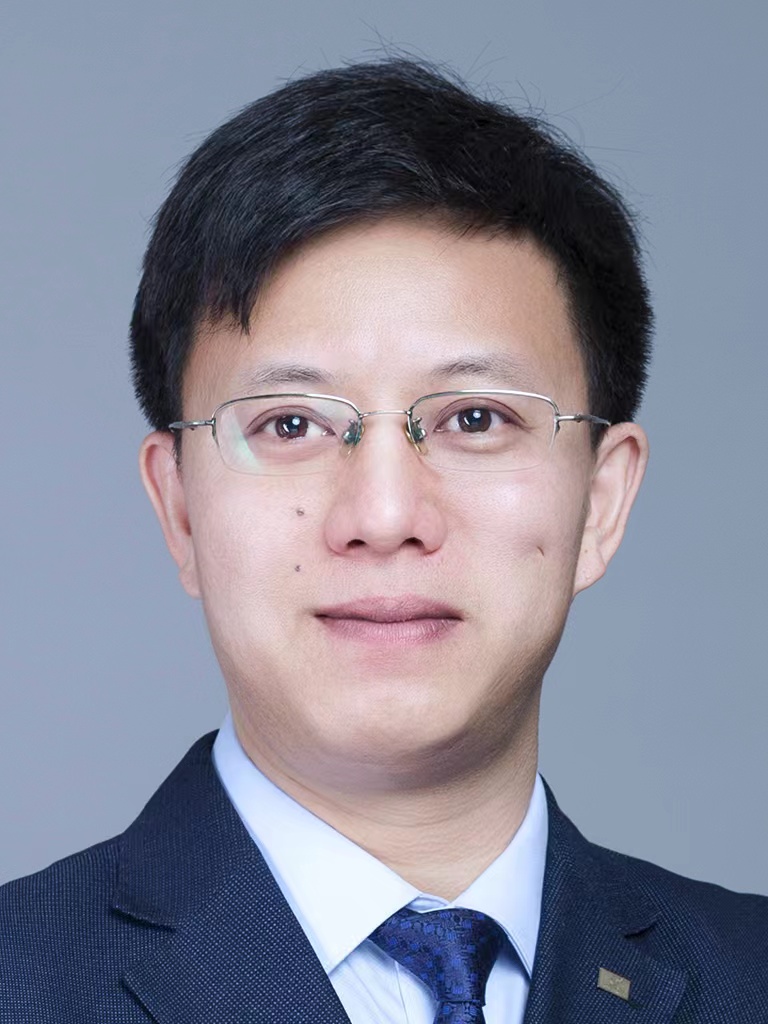}}
	\noindent {\bf Jun Fu}\
	He was a Postdoctoral Researcher with the Department of Mechanical Engineering, Massachusetts Institute of Technology (MIT), Cambridge, MA, USA, from 2010 to 2014. He is a Full Professor with Northeastern University, Shenyang, China. His current research is on dynamic optimization, optimal control, switched systems and their applications.
	Dr. Fu received the 2018 Young Scientist Award in Science issued by the Ministry of Education of China (the first awardee in Chinese Control Community). He is currently an Associate Editor for the Control Engineering Practice, the IEEE Transactions on Industrial Informatics, and the IEEE Transactions on Neural Networks and Learning Systems.}
\vspace{4\baselineskip}  
\end{document}